\documentclass[12pt,reqno]{amsart}
\usepackage{amsmath,amsthm}
\usepackage{caption}
\usepackage{mathtools}
\usepackage{cite}
\usepackage{enumitem}
\usepackage{color}
\usepackage{url}
\usepackage[dvipsnames]{xcolor}
\usepackage{graphicx}
\usepackage{xspace}
\usepackage[unicode=true,colorlinks]{hyperref}
\hypersetup{
	linkcolor=blue,
	citecolor=blue,
}
\usepackage{cleveref}
\usepackage{leftidx}
\usepackage[margin=1.2in]{geometry}
\usepackage{mathrsfs}

\newcommand{\Sp}{\operatorname{Sp}}

\newcommand{\GL}{\operatorname{GL}}
\newcommand{\SL}{\operatorname{SL}}

\newcommand{\Stab}{\operatorname{Stab}}

\newcommand{\Gr}{\operatorname{Gr}}
\usepackage{amssymb}

\newcommand{\defeq}{\mathrel{\mathop:}=}

\makeatletter
\renewcommand{\paragraph}{%
	\@startsection {paragraph}{4}
	{\z@} \z@ {-\fontdimen 2\font }\bfseries
}
\makeatother

\makeatletter 
\def\@cite#1#2{{\m@th\upshape\bfseries%
		[{#1\if@tempswa{\m@th\upshape\mdseries, #2}\fi}]}}
\makeatother 

\numberwithin{equation}{section}

\theoremstyle{plain}
\newtheorem{thm}{Theorem}[section]
\newtheorem{cor}[thm]{Corollary}
\newtheorem{prop}[thm]{Proposition}
\newtheorem{lem}[thm]{Lemma}

\theoremstyle{definition}
\newtheorem{definition}{Definition}[section]

\theoremstyle{remark}
\newtheorem{remark}{Remark}[section]

\newcommand{\ann}[1]{}

\newcommand{\DI}{\mathrm{DI}}

\newcommand\bfG{\mathbf{G}}
\newcommand\bfH{\mathbf{H}}

\newcommand\bfM{\mathbf{M}}

\newcommand\bfP{\mathbf{P}}
\newcommand\bfQ{\mathbf{Q}}

\newcommand\bfV{\mathbf{V}}
\newcommand\bfW{\mathbf{W}}

\newcommand\bfp{\mathbf{p}}
\newcommand\bfq{\mathbf{q}}

\newcommand\bfx{\mathbf{x}}
\newcommand\bbA{\mathbb{A}}

\newcommand\bbF{\mathbb{F}}

\newcommand\bbP{\mathbb{P}}
\newcommand\bbQ{\mathbb{Q}}
\newcommand\bbR{\mathbb{R}}

\newcommand\bbZ{\mathbb{Z}}
\newcommand\cA{\mathcal{A}}

\newcommand\cC{\mathcal{C}}
\newcommand\cD{\mathcal{D}}

\newcommand\cK{\mathcal{K}}
\newcommand\cL{\mathcal{L}}

\newcommand\cN{\mathcal{N}}
\newcommand\cO{\mathcal{O}}
\newcommand\cP{\mathcal{P}}

\newcommand\cS{\mathcal{S}}
\newcommand\cT{\mathcal{T}}

\newcommand\fg{\mathfrak{g}}

\newcommand\fp{\mathfrak{p}}
\newcommand\fq{\mathfrak{q}}
\newcommand\fs{\mathfrak{s}}
\newcommand{\fS}{\mathfrak{S}}
\newcommand\fz{\mathfrak{z}}

\newcommand{\R}{\mathbb{R}}
\newcommand{\C}{\mathbb{C}}
\newcommand{\Q}{\mathbb{Q}}
\newcommand{\Z}{\mathbb{Z}}
\newcommand{\N}{\mathbb{N}}

\DeclarePairedDelimiter\abs{\lvert}{\rvert}%
\DeclarePairedDelimiter\norm{\lVert}{\rVert}%

\makeatletter
\let\oldabs\abs
\def\abs{\@ifstar{\oldabs}{\oldabs*}}
\let\oldnorm\norm
\def\norm{\@ifstar{\oldnorm}{\oldnorm*}}
\makeatother

\newcommand{\eps}{\varepsilon}
\newcommand{\de}{\mathrm{d}}

\newcommand{\at}{\widetilde{\alpha}}

\newcommand{\pifix}{\pi_{\mathrm{fix}}}
\newcommand{\btt}{\widetilde{\beta}_\Omega}
\newcommand{\bN}{\beta_\cN}

\renewcommand{\Re}{\operatorname{Re}}
\renewcommand{\Im}{\operatorname{Im}}

\DeclareMathOperator{\diag}{diag}

\DeclareMathOperator{\Lie}{Lie}

\DeclareMathOperator{\Mat}{Mat}

\DeclareMathOperator{\Ad}{Ad}
\DeclareMathOperator{\Res}{Res}

\title{Birkhoff genericity on affine subspaces in horospheres}
\author[Shah]{Nimish~A.~Shah}
\address{The Ohio State University, Columbus, OH 43210}
\email{shah@math.osu.edu}
\author[Yang]{Pengyu~Yang}
\address{Morningside Center of Mathematics, Academy of Mathematics and Systems Science, Chinese Academy of Sciences, Beijing 100190}
\email{yangpengyu@amss.ac.cn}

\thanks{Pengyu Yang is supported by National Key R\&D Program of China 2022YFA1007500 and NSFC grant 22AAA00245.}

\begin{document}

\begin{abstract}
We study Birkhoff genericity for a simple uniformly expanding diagonal flow on
\(\mathrm{SL}_{n+1}(\mathbb R)/\mathrm{SL}_{n+1}(\mathbb Z)\), with initial points
restricted to affine subspaces of the expanding horospherical orbit through the identity coset. We prove that almost every point on such an affine subspace
is Birkhoff generic, except possibly in two situations: either the
defining matrix of the affine subspace has Diophantine exponent at least \(n\), or the affine subspace
is arbitrarily well approximable by affine subspaces of dimension \(r-1\)
defined over a real number field of degree \(m\ge 2\), with \(n+1=mr\).
As applications, we obtain Dirichlet non-improvability and logarithmic density
results for almost every point on these affine subspaces.
\end{abstract}

	\maketitle
	\tableofcontents
	\section{Introduction}

Let $X=G/\Gamma$ be a finite-volume homogeneous space and let
$\{a_t\}_{t\in\R}$ be a one-parameter diagonal subgroup of $G$.
A point $x\in X$ is said to be Birkhoff generic for the $\{a_t\}$-action with respect to the
$G$-invariant probability measure $\mu_X$ if
\[
\lim_{T\to\infty}\frac{1}{T}\int_0^T f(a_t x)\,\de t
=
\int_X f\,\de\mu_X,\qquad \forall f\in C_c(X).
\]

Suppose that the $a_t$-action is ergodic with respect to $\mu_X$. By the Birkhoff ergodic theorem, $\mu_X$-almost every point is Birkhoff generic. Let $P_a^-=\{g\in G:\lim_{t\to\infty} a_tga_{-t} \text{ exists in } G\}$ denote the stable subgroup for $\{a_t\}$. Then if $x$ is generic then $gx$ is also generic for any $g\in P_a^-$. Let $U_a^+\defeq \{g\in G: \lim_{t\to\infty} a_{-t}ga_t\to e\}$ denote the expanding horospherical subgroup for $\{a_t\}$. Then $P_a^-U_a^+$ is open in $G$. Therefore, for any $x\in X$, we have that $gx$ is generic for almost every $g\in U_a^+$ with respect to the Haar measure on $U_a^+$. We are interested in finding out when  the same conclusion holds for other measures on $U_a^+$, for example smooth measures on proper submanifolds of $U_a^+$. 

The above question was motivated by another class of similar results, where one shows that translates of certain smooth measures on a submanifold of a horosphere by a diagonal action get equidistributed in the limit \cite{KM98, Shah:Dirichlet,Yan20Invent}. These type of results are like mixing of piece of a low-dimensional submanifold under translates. While in this article we are interested in ergodic averages along the diagonal orbit from various points of the submanifold.   

Shi~\cite{Shi20} carefully studied this problem for subgroup action of semisimple groups $H$ on $G/\Gamma$, and established Birkhoff genericity for points on orbits of horospheres of $H$ for actions of $\R$-diagonalizable subgroups. For points on horospheres in the case of actions on the space of unimodular lattices, Kleinbock, Shi, and Weiss \cite{KSW17} have also obtained precise error terms for ergodic averages. See also \cite{FSU18, Kha20, Zha21} for progress on different cases of this problem.

More recently, a strong result was proved by Solan and Wieser \cite{SolanWeiser25} who established Birkhoff genericity for non-degenerate curves on horospheres of simple uniformly expanding diagonal group actions on the space of unimodular lattices. 

In this paper we consider the Birkhoff genericity question for points on affine subspaces of the expanding horospherical subgroup. Our work complements the results in \cite{Shi20, KSW17,EinsiedlerShi19, SolanWeiser25} and we refer the readers to the references therein for other applications, history and background in a more general context. 

	\subsection{Main results}
	Let $G=\SL_{n+1}(\R)$, $\Gamma=\SL_{n+1}(\Z)$, $X=G/\Gamma$, the space of unimodular lattices in $\R^{n+1}$, $\mu_X$ denote the $G$-invariant probability measure on $X$, and $x_0=e\Gamma/\Gamma\in X$. Let 
    \[
    t\mapsto a_t=\begin{psmallmatrix}
	e^{\frac{n}{n+1}t} & \\
	& e^{-\frac{1}{n+1}t}I_n
	\end{psmallmatrix}
    \]
    denote the uniformly expanding simple diagonal one-parameter subgroup. And let 
    \[
    U_a^+=\left\{\begin{psmallmatrix}
	1 & x_1 & \cdots & x_n \\
	& 1 &        &   \\
	&   & \ddots &   \\
	&   &        & 1 \\
	\end{psmallmatrix}:(x_1,\ldots,x_n)\in\R^n\right\},
    \]
    denote the expanding horospherical subgroup of $G$ with respect to $a_1$. From now on we will identify $U^+_a$ with $\R^n$. Let $e_1,\dots,e_n$ denote the basis of $U_a^+\cong\R^n$. Let $\mathcal{A}$ be a (proper) affine subspace of $U_a^+$ of dimension $d<n$. To parametrize $\mathcal{A}$, by permuting the coordinates, without loss of generality we may assume that the projection of $\mathcal{A}$ to the coordinate plane spanned by $e_1,\dots,e_d$ is bijective. Then, there exists a unique matrix $A\in\Mat_{d+1,n-d}(\R)$ such that 
    \begin{equation}
    \label{eq:LsubA}
    \cA=\cL_A\defeq\{(\bfx,A\widetilde{\bfx})\colon \bfx\in\R^d\}\text{, where $\widetilde{\bfx}=(1,\bfx)$ for all $\bfx\in\R^d$.}
    \end{equation}
    Here $A$ is called the matrix defining $\cA$. Let 
    \begin{equation} \label{eq:uA}
        u_A=\begin{psmallmatrix}
	I_{d+1} & A \\
	& I_{n-d} \\
	\end{psmallmatrix} \text{ and } x_A=u_Ax_0.
    \end{equation} 
    
	For $A\in\Mat_{d+1,n-d}(\R)$, we define the Diophantine exponent $\omega(A)$ of $A$ to be the supremum of $\omega>0$ such that the inequality $\norm{A\bfq+\bfp}\leq \norm{\bfq}^{-\omega}$ has infinitely many nonzero integral solutions $(\bfq,\bfp)\in\Z^{n-d}\times\Z^{d+1}$. By Dirichlet's approximation theorem, $\omega(A)\geq \frac{n-d}{d+1}$ for every $A$, and  the equality holds for almost every $A$. 
	
	We shall prove the following point-wise non-escape of mass result.
	\begin{thm} \label{thm:main_nondivergence}
		Suppose $\omega(A)<n$. Then for almost every $u\in\mathcal{A}$, every weak-$\ast$ limit point of $\{\frac{1}{T}\int_{0}^{T}a_tu\delta_{x_0}\de t\}_{T\to\infty}$ is a probability measure on $X$.
	\end{thm}
	
	\begin{remark}
		We note that the condition on the Diophantine exponent is optimal in the following sense. Suppose $\omega(A)>n$. Then by the argument in the proof of \cite[Lemma~6.5]{SY24PLMS}, see also \cite[Theorem~1.1]{KSSY}, escape of mass occurs for every $u\in\mathcal{A}$. More precisely, there exists $\eps_0>0$ and $T_i\to\infty$ such that  for every compact subset $K$ of $X$ and every $u\in\mathcal{A}$, the Lebesgue measure of the set $\{0\leq t\leq T_i\colon a_tux_0\in K\}$ is at most $(1-\eps_0)T_i$ for all $i\gg1$. 
        
        In the critical case of $\omega(A)=n$, we are not able to draw any conclusions. 
	\end{remark}

    In this article, a real number field is defined to be a subfield of $\R$ which is a finite extension of $\Q$; it is not necessarily totally real. Let $\Gr(k_1,k_2)$ denote the Grassmannian variety parametrizing $k_1$-dimensional linear subspaces of a $k_2$-dimensional vector space.
    
    \begin{definition}[$\bbF$-Liouville subspaces]  \label{def:F-Liouville}
        Let $\bbF$ be a real number field and $k \leq l < n$ are positive integers. We say a $k$-dimensional subspace $\cS$ of $\bbP^n(\R)$ is $(\bbF, l)$-\emph{Liouville} if for every $\kappa\geq 1$, there exists an $l$-dimensional subspace $\cT$ of $\bbP^n(\R)$ which is defined over $\bbF$, such that
        \begin{equation}
            d(\cS, \cT)\leq H_{\cL}(\cT)^{-\kappa}.
        \end{equation}
        Here $d(\cS, \cT):=\max_{x\in\cS}\min_{y\in\cT}d(x,y)$, where $d(x,y)$ is the angle between the lines in $\R^{n+1}$ that $x$ and $y$ parametrize.
        The height function $H_\cL$ is an exponential Weil height function $H_\cL:\Gr(l+1,n+1)(\bbF)\to\R_{>0}$ associated with the ample generator $[\cL]$ of the Picard group $\operatorname{Pic}(\Gr(l+1,n+1))\cong\Z$. Note that the height function is determined by $[\cL]$ up to multiplication by a constant, but changing the constants does not affect our definition of Liouville-ness. 
        
        We say an affine subspace of $\bbA^n$ is $(\bbF,l)$-Liouville if its projective closure in $\bbP^n$ is.
    \end{definition}

    \begin{remark} \label{rem:H-dim-F-Liouville}
        Let $\cL_{k;\bbF,l}$ denote the set of $(\bbF,l)$-Liouville subspaces in $\Gr(k+1,n+1)$. Then \Cref{def:F-Liouville} gives natural coverings of $\cL_{k;\bbF,l}$ by thin neighborhoods of $\bbF$-Schubert cycles in $\Gr(k+1,n+1)$ defined by the condition $\cS\subset\cT$. These Schubert cycles are of dimension $(k+1)(l-k)$. From this one can show that the Hausdorff dimension of $\cL_{k;\bbF,l}$ is $(k+1)(l-k)$, which is strictly smaller than the dimension of $\Gr(k+1,n+1)$. Since this fact is not used in this article, we leave it to the reader to verify the details.
    \end{remark}
    
	The following is the main result of this article.
	
	\begin{thm}  \label{thm:main_equidistribution}
		Suppose $\omega(A)<n$. Then one of the following possibilities holds:
		\begin{enumerate}
			\item \label{item:equidistribution} For almost every $u\in\mathcal{A}$, we have
			\begin{equation*}
			\lim_{T\to\infty}\frac{1}{T}\int_{0}^{T}f(a_tux_0)\de t =
			\int_{X} f \de\mu_{X},\quad\forall f\in C_c(X).
			\end{equation*}
			\item \label{item:algebraic obstruction} There exist a real number field $\bbF$ of degree $m\geq 2$ and a positive integer $r\geq d+1$ such that
			\begin{itemize}
				\item $rm=n+1$,
				\item $\mathcal{A}$ is $(\bbF, r-1)$-Liouville.
			\end{itemize}
		\end{enumerate}
        In particular, if $n+1$ is prime, then possibility~(2) does not occur. 
	\end{thm}

\begin{remark} \label{rem:AinT}
    Suppose that $\mathcal{A}$ is contained in a $(r-1)$-dimensional affine subspace defined over $\bbF$ such that $rm=n+1$; see \Cref{cor:inFSubspace}. Then by \cite[Proof of Lemma~6.6]{SY24PLMS}, there exists a proper closed subset $\cK$ of $X$ such that $a_t\mathcal{A}x_0\subset \cK$ for all $t\geq 0$. So for every $u\in\cA$, the sequence of probability measures $\frac{1}{T_i}\int_{0}^{T_i}a_tu\delta_{x_0}\de t$ do not equidistribute in $X$ for any sequence $T_i\to\infty$. 
    \end{remark}

\subsubsection{Geometric formulation}
Let $W$ be a linear subspace of $\bbP^n(\R)$. As done in the `geometric formulation' given in the introduction of \cite{SY24PLMS}, we define the Diophantine exponent $\omega(W)$ of $W$ to be the supremum of $\omega$ such that the following holds: there exist infinitely many hyperplanes $Q\subset \bbP^n(\R)$ defined over $\Q$ such that $d(W, Q) \leq H(Q)^{-\omega-1}$. Here $d(W,Q)=\sup_{[x]\in W}\inf_{[y]\in Q}\frac{\norm{x\wedge y}}{\norm{x}\norm{y}}$, and $H$ denotes the Weil height on $\Gr(n,n+1)(\Q)\cong \bbP^n(\Q)$ associated with the line bundle $\mathcal{O}_{\bbP^n}(1)$. 

In the above notation, we have that $\omega(\cL_A)=\omega(A)$; see~\cite{SY24PLMS}.

\begin{thm}\label{thm:geom-equi} 
Let $\lambda$ be a finite measure on $G$ whose pushforward on $P^-_a\backslash G\cong \bfP^{n}(\R)$ is absolutely continuous with respect to the Lebesgue measure on $W$. Suppose that $\omega(W)<n$ and $W$ is not $(\mathbb{F},(r-1))$-Liouville for a real number field $\bbF$ of degree $m\geq 2$ such that $n+1=mr$. Then for $\lambda$-almost all $g\in G$, we have
\[
\frac{1}{T}\int_{0}^T f(a_tgx_0)\de t\to \int_{X} f\de \mu_{X}, \forall f\in C_c(X)
\]
\end{thm}
We leave it to the reader to verify that \Cref{thm:geom-equi} can be deduced from \Cref{thm:main_equidistribution}.

	\subsection{Dirichlet non-improvability on subspaces} 

    Initial motivation for our study comes from the Diophantine approximation. {Denote by $\|\cdot\|$ the supremum norm on $\bbR^n$, where $n\geq 1$ (unless specified otherwise, all the norms on $\bbR^n$ will be taken to be the supremum norm).} Let $0<\delta\leq 1$. {Following} Davenport and Schmidt \cite{DS6970}, let $\DI(\delta)$ denote the set of vectors $\bfx\in\bbR^n$ such that for all large $T\geq 1$,
 \begin{equation}\label{eq:lf}
	\exists (p,\bfq)\in \bbZ\times (\bbZ^n\setminus \{0\})\text{ such that }	\begin{cases}
		\abs{\bfx\cdot\bfq+p}\leq \delta T^{-n} \\
		\|\bfq\| \leq  \delta T.\end{cases}
		\end{equation}
		Similarly, let $\DI'(\delta)$ denote the set of vectors $\bfx\in\bbR^n$ such that for all large $T\geq 1$, 
  \begin{equation}\label{eq:vect}
      \exists (\bfp,q)\in \bbZ^n \times  (\bbZ\setminus\{0\}) \text{ such that }   \begin{cases}
		\|q\bfx+\bfp\|\leq \delta T^{-1} \\
		\abs{q} \leq  \delta T^n.
		\end{cases}
		\end{equation}

	By Dirichlet's approximation theorem, $\DI(1)=\bbR^n$ and $\DI'(1)=\bbR^n$. 
    
    Let $\DI=\cup_{0<\delta<1}\DI(\delta)$ and $\DI'=\cup_{0<\delta<1}\DI'(\delta)$, denote the sets of Dirichlet's improvable vectors. Davenport and Schmidt~\cite{DS6970} proved that the sets $\DI$ and $\DI'$ are Lebesgue null in $\R^n$, see also \cite{KW08}. By the arguments as in \cite{KW08,SW17} involving Dani-correspondence, it is straightforward to deduce the following result from \Cref{thm:main_equidistribution}. 

    \begin{thm}
        \label{thm:Dirichlet}
        Let $\cA$ be a proper $d$-dimensional affine subspace of $\R^{n}$ that does not satisfy the condition in possibility~(\ref{item:algebraic obstruction}) of \Cref{thm:main_equidistribution}, and suppose that $\omega(\cA)<n$. Then $\DI\cap\cA$ and $\DI'\cap \cA$ are Lebesgue null in $\cA$.  
    \end{thm}

    The result weakens the conditions for Dirichlet non-improvability in \cite[Theorem~1.3]{SY24PLMS} for the case of $d=1$, $n\geq 3$ and $n+1$ is even (for $\cT=\N$ in its statement). 

     Following \cite[Theorem~1.1]{SolanWeiser25} we formulate a density result on improvability of Dirichlet's approximation on affine subspaces of $\R^n$. For any $\bfx\in\R^d$ and $\delta\in (0,1)$, let
    \begin{align*}
    \overline\cD_{\bfx}(\delta)
    &=\limsup_{N\to\infty} \frac{1}{\log N}\sum_{\substack{T\in\{1,\ldots,N\}:\\ \text{\eqref{eq:lf} holds for $T$}}} \frac{1}{T}, \\
    \underline\cD_{\bfx}(\delta)
    &=\liminf_{N\to\infty} \frac{1}{\log N}\sum_{\substack{T\in\{1,\ldots,N\}:\\ \text{\eqref{eq:lf} holds for $T$}}} \frac{1}{T},
    \end{align*}
    and similarly define $\overline{\cD'}_{\bfx}(\delta)$ and $\underline{\cD'}_{\bfx}(\delta)$, where $T$ in the corresponding sum satisfies \eqref{eq:vect}. Using the arguments of the proof of \cite[Theorem~1.1]{SolanWeiser25}, the following result is a direct consequence of \Cref{thm:main_equidistribution}.

    \begin{thm}
        \label{thm:dirichlet-density} There exist continuous
strictly increasing functions $f$ and $g$ from $[0, 1]$ to $[0, 1]$ with $f(0)=g(0)=0$ and $f(1)=g(1)=1$ with the following property: Let $\cA$ be an affine subspace of $\R^n$ that does not satisfy possibility~(\ref{item:algebraic obstruction}) of \Cref{thm:main_equidistribution}, and suppose that $\omega(\cA)<n$. Then for every $0<\delta<1$ and Lebesgue almost every $\bfx\in \cA$, we have 
\[
\overline\cD_{\bfx}(\delta)=\underline\cD_{\bfx}(\delta)=f(\delta) \text{ and } \overline{\cD'}_{\bfx}(\delta)=\underline{\cD'}_{\bfx}(\delta)=g(\delta). 
\]
    \end{thm}

\subsection{Strategy of the proof and organization of the paper}
The general strategy of the proof in this article is based on extending the techniques developed in \cite{Shi20}, \cite{PSS23}, and \cite{SY24TAMS}. In \Cref{sec:non-escape} we describe and recall properties of the Margulis height function on $G/\Gamma$ as defined in \cite{BQ12} for the action of $H_d=\SL_{d+1}(\R)$. We study the exponential growth rate of the height function along trajectories of a diagonal group $\{b_t\}$ commuting with $H_d$. Using the relation between the growth rate and Diophantine exponent obtained in \cite{SY24TAMS}, we define a modified height function with contraction property for the $\{a_t\}$-action. This allows us to prove \Cref{thm:main_nondivergence} in view of the methods developed in \cite{Shi20}.  In \Cref{sec:Sing}, we recall a result inspired by \cite{CE15} from \cite{Shi20} to obtain unipotent invariance of the limiting measures for almost all ergodic averages. These limiting measures are invariant under the simple diagonal group $\{a_t\}$ by definition. These observations allow one to use measure rigidity theorems of Ratner \cite{Ratner:measure} as done by Einsiedler and Shi \cite{EinsiedlerShi19}. They show that if such a limit measure is not $G$-invariant, then it must be positive on certain algebraic subvariety of $G$ projected to $G/\Gamma$; called singular manifolds. To analyze how the ergodic averages assigns mass to the singular manifold, we carefully construct a height function to measure `distance' from the singular submanifold. We ensure that this function satisfies nice properties, in particular it satisfies the contraction property with respect to the $a_t$-action. The contraction property is related to exponential growth of a new height function along the $b_t$-trajectories. This construction of the height function given in \Cref{sec:Sing} is a novel part of this article.  In \Cref{sec:obstruction} we interpret the high exponential growth rate in terms of the $\bbF$-Liouville property. 
    
	\section{Non-escape of mass}
	\label{sec:non-escape}
	In this section, we establish pointwise nonescape of mass on affine subspaces with small Diophantine exponents. We follow the strategy of \cite[Section 4]{Shi20}. Many of the results in this section have been proven in \cite{Shi20} and \cite{SY24TAMS}, and we provide some details for the readers' convenience.

    \subsection{Notation} \label{subsec:notation}
	Let $G=\SL_{n+1}(\bbR)$ and 
    $H_d=\begin{psmallmatrix}
	\SL_{d+1}(\bbR) & \\
	& I_{n-d}
	\end{psmallmatrix}$, where $I_k$ denotes the $k\times k$-identity matrix. Consider the one-parameter diagonal subgroups of $G$ defined by the following: For all $t\in\R$,
	\begin{align}
	&a_t=\begin{psmallmatrix}
	e^{\frac{n}{n+1}t} & \\
	& e^{-\frac{1}{n+1}t}I_n
	\end{psmallmatrix},\label{eq:gt}\\
	&b_t=\begin{psmallmatrix}
	e^{\frac{n-d}{(d+1)(n+1)}t}I_{d+1} & \\
	& e^{-\frac{1}{n+1}t}I_{n-d}
	\end{psmallmatrix}\in Z_G(H_d)\text{, and } \label{eq:bt}\\
	&c_t=\begin{psmallmatrix}
	e^{\frac{d}{d+1}t} & & \\
	& e^{-\frac{1}{d+1}t}I_d & \\
	& & I_{n-d}
	\end{psmallmatrix}\in H_d. \label{eq:ct}
	\end{align}
	Then $a_t=b_tc_t$. With this normalization, the eigenvalues of $\Ad(a_t)$ on the Lie algebra of $U_a^+$ and the eigenvalues of $\Ad(c_t)$ on the Lie algebra of $U_a^+\cap H_d$ are all equal to $e^t$. 

    Let $u:\R^d\to U\defeq U_a^+\cap H_d$ be the map defined as
    \begin{equation}
    \label{eq:us}
    u(s)=\begin{psmallmatrix}
        1& s &0\\
        & I_{d}\\
        && I_{n-d}
    \end{psmallmatrix}\in U=U_a^+\cap H_d, \quad \forall s\in \R^d.
    \end{equation}

    Given $A\in\Mat_{d+1,n-d}(\R)$, we write $A$ in the block form
	\[
	A=\begin{psmallmatrix}
	A_1 \\
	A_2
	\end{psmallmatrix} \text{, where $A_1\in\operatorname{Mat}_{1,n-d}(\R)$ and $A_2\in\operatorname{Mat}_{d,n-d}(\R)$.}
	\]
	 Let
    $\mathcal{A}=\{(s,\tilde sA)\colon \bfx\in\R^d\}$, where $\tilde s=(1,s)\in\R^{d+1}$ for every $s\in\R^d$.
    We identify $\cA$ with an affine subspace of $U_a^+=\R^n$.  Let
	\begin{equation} \label{eq:z_A}
	z_A=\begin{psmallmatrix}
	1   &    &       \\
	    & I_d & -A_2 \\
	    &     & I_{n-d}
	\end{psmallmatrix}\in Z_G(\{a_t\}).
	\end{equation}
    Now for every $s\in\R^d$, we have 
	\begin{equation*} 
	\begin{psmallmatrix} 1&s&A\widetilde{s}\\&I_d\\&&I_{n-d}\end{psmallmatrix}= z_Au(s)u_A\in\cA\subset U_a^+.
    \end{equation*}
    Let  $x_A=u_Ax_0$. Then, for any $t\in\R$ and $s\in\R^d$,
    \begin{equation} \label{eq:basic identity}
    a_t\begin{psmallmatrix} 1&s&A\widetilde{s}\\&I_d\\&&I_{n-d}\end{psmallmatrix}x_0=z_Aa_tu(s)x_A.
	\end{equation}
Therefore, since $\mu_X$ is $z_A$-invariant, the possibility~(\ref{item:equidistribution}) can be equivalently expressed as: For Lebesgue almost all $s\in\R^d$,
\begin{equation}
    \label{eq:equidistribution}
    \lim_{T\to\infty}\frac{1}{T}\int_{0}^{T}f(a_tu(s)x_A)\de t =
			\int_{X} f \de\mu_{X},\quad\forall f\in C_c(X).
\end{equation}
    \subsection{Height function with respect to the point at infinity} \label{subsect:margulis function w.r.t. infty}
	Restricting the standard action of $G$ on $\bbR^{n+1}$ to $H_d$, we have the following decomposition of $H_d$-modules
	\begin{equation*}
	\bbR^{n+1}=V_0^\perp\bigoplus V_0,
	\end{equation*} 
	where $V_0^\perp$ is the $\bbR$-span of $\{e_0,\dots,e_d\}$ and $V_0$ is the $\bbR$-span of $\{e_{d+1},\dots,e_{n}\}$. Here, $V_0^\perp$ is the standard representation of $H_d\cong \SL_{d+1}(\R)$ and $H_d$ acts trivially on $V_0$.
	For $0\leq k\leq n+1$, taking the $k$-exterior products we get
	\begin{equation}\label{eq:decomposition of exterior representations}
	\bigwedge^k\bbR^{n+1}=
	\bigoplus_{i=\max\{0,k-(n-d)\}}^{\min\{d+1,k\}}
	\bigwedge^i(V_0^\perp)\bigotimes\bigwedge^{k-i}V_0,
	\end{equation}
	because if $\bigwedge^i(V_0^\perp)\neq 0$, then $0\leq i\leq d+1$, and if $\bigwedge^{k-i}V_0\neq 0$, then $0\leq k-i\leq n-d$.  
	
	We recall the construction of a Margulis height function $\alpha:X\to [0,+\infty]$ from \cite{BQ12}. We shall specify it in our setting, and optimize the constants. 
	In view of \eqref{eq:decomposition of exterior representations}, for $\max\{0,k-(n-d)\}\leq i\leq\min\{d+1,k\}$, let $\pi_i$ denote the projection
	 \begin{equation} \label{eq:pi-i}
     \pi_i:\bigwedge^k\bbR^{n+1}\longrightarrow\bigwedge^i(V_0^\perp)\bigotimes\bigwedge^{k-i}V_0.
	\end{equation}
	Since $V_0$ and $V_0^\perp$ are $N_G(H_d)$-invariant, each $\pi_i$ is $N_G(H_d)$-equivariant. 
    
    Now
    \[
\left(\bigwedge^0(V_0^\perp)\bigotimes\bigwedge^{k}V_0\right)\bigoplus\left(\bigwedge^{d+1}(V_0^\perp)\bigotimes\bigwedge^{k-(d+1)}V_0\right)
    \]
    is the space of $H_d$-fixed $k$-vectors. Let $\pifix=\pi_0\oplus\pi_{d+1}$. 
	
	We take $\delta_k=(n+1-k)k$ for $0\leq k\leq n+1$. Let $\eps>0$ and $0<k<n+1$. For every $v\in\bigwedge^k\bbR^{n+1}$, we let
	\begin{equation}\label{eq:def_phi_eps}
	\varphi_{\eps}(v)=\begin{cases}
	\min_{1\leq i\leq d}\eps^{\frac{d+1}{d-i+1}\delta_k}\norm{\pi_i(v)}^{-\frac{d+1}{d-i+1}}, & \text{if } \norm{\pifix(v)}<\eps^{\delta_k},\\
	1, & \text{otherwise}.
	\end{cases}
	\end{equation}
	
	Given $\theta>0$ and $\eps>0$, for $y\in X$ we define
	\begin{equation}\label{eq:alpha_eps^theta}
	\alpha_\eps^{\theta}(y)=\max_v \varphi_\eps^\theta(v)\in [1,+\infty],
	\end{equation}
	where $v$ varies over all nonzero $y$-integral decomposable vectors in        $\cup_{k=1}^{n}\bigwedge^k\bbR^{n+1}$; that is if $y=g\Z^{n+1}$ for some $g\in\SL_{n+1}(\R)$, then $v=g(v_1\wedge \cdots \wedge v_k)$ for some $1\leq k\leq n$ and linearly independent $v_1,\ldots,v_k\in\Z^{n+1}$. 

    To simplify notation, we write $\alpha_\eps:=\alpha_\eps^1$. Note that $\alpha_\eps^\theta=(\alpha_\eps)^\theta$.
    
    \begin{remark} \label{rem:alpha-eps}
        \begin{enumerate}
            \item 
        We note that, if $\alpha_\eps^\theta(x)<\infty$, then $\alpha_\eps^\theta(gx)<\infty$ for all $g\in H_d$. In fact, $\alpha_\eps^\theta$ is Lipschitz with respect to the action of $H_d$; see~\cite[Lemma 4.1]{Shi20}.

\item For any compact set $K\subset X$, there exists $\eps\in(0,1)$ such that $\sup_{y\in K}\alpha_\eps^\theta(y)<\infty$.
 \end{enumerate}

    \end{remark}

\subsection{Exponential growth rate of height function along \texorpdfstring{$b_t$}{b\_t}-trajectory}

	For each $x\in X$, 
    we define the exponential growth rate of the trajectory $\{b_tx\}_{t\geq0}$ to be
	\begin{equation} \label{eq:definition of rho}
    \rho_{\eps,\theta}(x)\defeq  \limsup_{t}\frac{\log \alpha^\theta_\eps(b_tx)}{t}\geq 0.
	\end{equation}

\begin{remark}
  \label{rem:rhoxA}
     Let $A\in\Mat_{d+1,n-d}(\R)$. Let $x_A=u_Ax_0$ be as defined in \eqref{eq:uA}. By \cite[Proposition 4.1]{SY24TAMS}, if $\omega(A)<n$, then $\rho_{\eps,\theta}(x_A)<\theta$.   
\end{remark}

Let $P_b^-$ denote the stable parabolic subgroup associated to $\{b_t\}$; that is, 
\begin{equation} \label{eq:Pb-}
P_b^-=\{\omega\in G: \lim_{t\to\infty} b_t\omega b_t^{-1}\in G \text{ exists}\}.
\end{equation}

We note that $N_G(H_d)=Z_G(\{b_t\})\subset P_b^{-}$. Also, $Z_G(\{a_t\}\{u(s)\}_{s\in\R^d})\subset P_b^-$, see \Cref{lem:ZGF-Pb-}. 

\begin{lem} \label{lem:rhoPb-} Let $x\in X$. If $\rho_{\eps,\theta}(x)\leq\theta$, then
\[
\rho_{\eps,\theta}(\omega x)=\rho_{\eps,\theta}(x),\,\forall \omega\in P_b^-.
\]
\end{lem}

More generally, let $\Omega\subset P^-_b$ be a compact set. We define 
\begin{equation} \label{eq:rho-Omega}
\rho_{\eps,\theta,\Omega}(x)=\limsup_{t\to\infty} \frac{\sup_{\omega\in\Omega} \log\alpha^\theta_\eps(b_t\omega x)}{t}.
\end{equation}
Then $\rho_{\eps,\theta}(x)=\rho_{\eps,\theta,\{e\}}(x)$. We will prove the following uniform version of \Cref{lem:rhoPb-}. 

\begin{prop}
    \label{prop:bt-Omega}
    Suppose $\rho_{\eps,\theta}(x)\leq \theta$. Then
    $\rho_{\eps,\theta,\Omega}(x)=\rho_{\eps,\theta}(x)$.
\end{prop}

These results are needed only in \Cref{sec:obstruction}, so the proof may be skipped in the first reading.

\begin{proof}
It is enough to prove that $\rho_{\eps,\theta,\Omega}(x)\leq \rho_{\eps,\theta}(x)$.  So, we assume $\rho_{\eps,\theta,\Omega}(x)>0$ and let $0<\rho_0<\min\{\theta,\rho_{\eps,\theta,\Omega}(x)\}$. So it is enough to prove that $\rho_{\eps,\theta}(x)\geq \rho_0$.

   The arguments given here are based on \cite[Section 4]{SY24TAMS}. Write $x=gx_0$ for some $g\in G$. Then, there exist sequences $t_m\to\infty$, and $\{\omega_m\}\subset\Omega$,  such that 
   \[
   \rho_0<\rho_{\eps,\theta,\Omega}(x)=\lim_{m\to\infty} \frac{\log\alpha_{\eps}^\theta(b_{t_m}\omega_mx)}{t_m}.
   \]
   In view of the definitions \eqref{eq:def_phi_eps} of $\varphi_\eps$ and \eqref{eq:alpha_eps^theta} of $\alpha_\eps^\theta$, after passing to a subsequence, there exists a $1\leq k\leq n$, and a sequence of nonzero decomposable $k$-vectors $v_m\in\bigwedge^k\Z^{n+1}$ such that 
   \[
  \alpha_{\eps}^\theta(b_{t_m}\omega_m x)=(\varphi_\eps(b_{t_m}\omega_m g v_m))^{\theta}.
   \]
   
   We recall that $\pi_l:\bigwedge^k\R^{n+1}\to W_l\defeq \bigwedge^l V_0^\perp\oplus\bigwedge^{k-l}V_0$ for all $0\leq l\leq d+1$, see \eqref{eq:pi-i}. Since $\rho_0>0$, by \eqref{eq:def_phi_eps}, 
   \begin{equation}
       \label{eq:fix}
       \norm{\pifix(b_{t_m}\omega_m gv_m)}<\eps^{\delta_k},
   \end{equation}
   where $\pifix=\pi_0+\pi_{d+1}$ and $\delta_k=(n+1-k)k$, and
   \[
   e^{\theta^{-1}\rho_0t_m}\leq \varphi_\eps(b_{t_m}\omega_m g v_m)=\min_{1\leq l\leq d} \eps^{\frac{d+1}{d+1-l}\delta_k}\norm{\pi_l(b_{t_m}\omega_m g v_m)}^{-\frac{d+1}{d+1-l}},
   \]
   equivalently,
   \begin{equation}
       \label{eq:pi-l}
        \norm{\pi_l(b_{t_m}\omega_m gv_m)}\leq \eps^{\delta_k} e^{-(1-\frac{l}{d+1})\theta^{-1}\rho_0t_m},
        \quad \forall 1\leq l\leq d.
    \end{equation}

Since $\Omega\subset P_b^-$, we have
\begin{equation} \label{eq:C1}
    C\defeq \sup_{\omega\in\Omega}\sup_{t\geq 0}\norm{b_t\omega^{-1}b_t^{-1}}<\infty.
\end{equation}

We note that by \eqref{eq:bt}, $b_t$ acts on $W_l$, the image of $\pi_l$, as a scalar multiplication by 
\begin{equation} \label{eq:bt-action}
e^{\bigl(\frac{l}{d+1}-\frac{k}{n+1}\bigr)t}, \, \forall 0\leq l\leq d+1.
\end{equation}
In particular, 
\begin{equation} \label{eq:B-Wl}
     \omega W_l\subset W_l\oplus \cdots \oplus W_0 \text{, $\forall \omega\in P^-_b$}.
\end{equation}

    For each $m\in\N$, let 
    \[
    E_m\defeq \norm{\pi_0(b_{t_m}\omega_m gv_m)}.
    \]
    Then by passing to a subsequence, we may assume that 
    \begin{align} \label{eq:Em-case1}
        \text{either } & E_m\leq \eps^{\delta_k} e^{-\theta^{-1}\rho_0t_m}\text{, for all $m\in\N$}\\
        \text{or } & E_m>\eps^{\delta_k} e^{-\theta^{-1}\rho_0t_m} \text{, for all $m\in\N$.} \label{eq:Em-case2}
    \end{align}
    
    \subsubsection*{Case 1} Suppose that \eqref{eq:Em-case1} holds. Let $m\in \N$. Let $t_m'=t_m-\theta^{-1}\rho_0t_m$. Then, by \eqref{eq:pi-l} and \eqref{eq:bt-action}, we get
    \[
    \norm{\pi_l(b_{t_m'}\omega_m gv_m)}\leq \eps^{\delta_k} e^{-\bigl(1-\frac{k}{n+1}\bigr)\theta^{-1}\rho_0t_m},\quad \forall 0\leq l\leq d+1. 
    \]
    Hence,
    \begin{equation} \label{eq:all-same-tmprime}
    \norm{b_{t_m'}\omega_m gv_m}\leq \eps^{\delta_k} e^{-\bigl(1-\frac{k}{n+1}\bigr)\theta^{-1}\rho_0t_m}.
    \end{equation}
    As $\theta^{-1}\rho_0<1$, we have $t_m'\geq 0$. Hence by \eqref{eq:C1} applied to $\bigwedge^k\R^{n+1}$, we get
    \[
    \norm{b_{t_m'} gv_m}=\norm{(b_{t_m'}\omega_m^{-1}b_{t_m'}^{-1})b_{t_m'}\omega_m gv_m}\leq C^k\eps^{\delta_k} e^{-(1-\frac{k}{n+1})\theta^{-1}\rho_0t_m}.
    \]
    Now going back from $t_m'$ to $t_m=t_m'+\theta^{-1}\rho_0t_m$, in view of \eqref{eq:bt-action}, we get
    \begin{align*}
       \norm{\pi_0(b_{t_m}gv_m)}&\leq C^k\eps^{\delta_k}e^{-\theta^{-1}\rho_0t_m},\\
       \norm{\pi_{d+1}(b_{t_m}gv_m)}&\leq C^k\eps^{\delta_k},\\ 
        \norm{\pi_l(b_{t_m}gv_m)}&\leq C^k\eps^{\delta_k} e^{-(1-\frac{l}{d+1})\theta^{-1}\rho_0t_m}\text{, $\forall 1\leq l\leq d$.}
    \end{align*}

    Pick $s_m=-\sqrt{t_m}$ and $t_m''=t_m+s_m$. Then,
    \begin{align*} 
       \norm{\pi_0(b_{t_m''}gv_m)}&\leq C^k\eps^{\delta_k}e^{-\frac{k}{n+1}s_m}e^{-\theta^{-1}\rho_0t_m}\to 0\text{, as } s_m/t_m\to 0,\\
       \norm{\pi_{d+1}(b_{t_m''}gv_m)}&\leq C^ke^{(1-\frac{k}{n+1})s_m}\eps^{\delta_k}\to 0 \text{, as }s_m\to-\infty, \\ 
        \norm{\pi_l(b_{t_m''}gv_m)}&\leq (C^ke^{(\frac{l}{d+1}-\frac{k}{n+1})s_m})\eps^{\delta_k} e^{-\frac{d+1-l}{d+1}\theta^{-1}\rho_0t_m}\text{, $\forall 1\leq l\leq d$}.
    \end{align*}
So,
    \begin{equation}
        \label{eq:pifix1}
        \lim_{m\to\infty} \norm{\pifix(b_{t_m''}gv_m)}=0.  
    \end{equation}
Therefore, for all $m\gg 1$, we have
\[
\varphi_{\eps}^\theta(b_{t_m''}gv_m)=\min_{1\leq l\leq d+1} \eps^{\frac{d+1}{d+1-l}\delta_k}\norm{\pi_l(b_{t_m''}gv_m)}^{-\frac{d+1}{d+1-l}}\geq Ce^{cs_m}e^{\rho_0t_m},
\]
for some $C>0$ and $c\in\R$ independent of $m$. So, since $\lim_{m\to\infty} s_m/t_m\to 0$, we get
    \[
    \rho_{\eps,\theta}(x)\geq \limsup_{m\to\infty} \frac{\log\varphi_{\eps}^\theta(b_{t_m''}gv_m)}{t_m''}\geq \rho_0.
    \]
   This completes the proof in Case~1.
   
    \subsubsection*{Case 2} Suppose that \eqref{eq:Em-case2} holds. This is the most interesting case. We note that in this case $k\leq n-d$, because if $k>n-d$, then the image of $\pi_0$ is $\{0\}$. 

    Let $m\in\N$. Since $v_m$ is a decomposable vector, using Pl\"ucker relations as in \cite[(4.16)]{SY24TAMS}, we get the following: For all $1\leq l\leq d$,
\begin{equation}
    \norm{\pi_{l+1}(b_{t_m}\omega_mgv_m)}\norm{\pi_0(t_m\omega_mgv_m)}\leq k\norm{\pi_l(b_{t_m}\omega_mgv_m)}\norm{\pi_1(b_{t_m}\omega_mgv_m)}.
\end{equation}
    Inductively, for all $1\leq l\leq d+1$, we get
    \begin{align}
    \norm{\pi_{l}(b_{t_m}\omega_m gv_m)}&\leq k^{l-1}E_m^{-(l-1)}\norm{\pi_1(b_{t_m}\omega_mgv_m)}^{l} \notag \\
    &\leq k^{l-1}E_m^{-(l-1)}\eps^{l\delta_k} e^{-\frac{ld}{d+1}\theta^{-1}\rho_0t_m}.
    \label{eq:plucker-consequence}
    \end{align}
    
    Let $t_m'=t_m+(d\theta^{-1}\rho_0t_m+(d+1)\log E_m)$.
    So
    for all $0\leq l\leq d+1$,  by  we get
    \begin{align}
    \norm{\pi_l(b_{t_m'}\omega_m gv_m)}
    &\leq k^d\eps^{l\delta_k}E_m^{\bigl(1-\frac{k(d+1)}{n+1}\bigr)}e^{-\frac{kd}{n+1}\theta^{-1}\rho_0t_m}.
    \end{align}

    Since $\log E_m\geq -\theta^{-1}\rho_0t_m$ and $\theta^{-1}\rho_0<1$, we have $t_m'\geq 0$. Therefore, by \eqref{eq:C1} and \eqref{eq:B-Wl}, assuming $0<\eps<1$, for any $0\leq l\leq d+1$,
\begin{align*}
\norm{\pi_l(b_{t_m'}gv_m)}
&\leq \sup_{j\geq l} C^k\norm{\pi_j(b_{t_m'}\omega_mgv_m)}\\
&\leq C^k k^d \eps^{l\delta_k} E_m^{\bigl(1-\frac{k(d+1)}{n+1}\bigr)}e^{-\frac{kd}{n+1}\theta^{-1}\rho_0t_m}.
\end{align*}
Therefore, going back from $t_m'$ to $t_m$ using \eqref{eq:bt-action}, we get for all $0\leq l\leq d+1$,
\begin{align*}
\norm{\pi_l(b_{t_m}gv_m)}
&\leq C^kk^d\eps^{l\delta_k} E_m^{-(l-1)}e^{-\frac{ld}{d+1}\theta^{-1}\rho_0t_m}.
\end{align*}
Now by \eqref{eq:fix}, we have $E_m<\eps^{\delta_k}$. And by \eqref{eq:Em-case2},  $E_m^{-(l-1)}\leq \eps^{-(l-1)\delta_k}e^{(l-1)\theta^{-1}\rho_0t_m}$, and $(l-1)-\frac{ld}{d+1}=-(1-\frac{l}{d+1})$ for $1\leq l\leq d+1$. Therefore,
\begin{equation}  \label{eq:each-l-cases}
\norm{\pi_l(b_{t_m}gv_m)}\leq 
\begin{cases} 
 C^kk^d E_m &\text{ if $l=0$,}\\
 C^kk^d\eps^{(d+1)\delta_k} (E_m^{-1}e^{-\theta^{-1}\rho_0t_m})^d\leq C^k k^d\eps^{\delta_k} &\text{ if $l=d+1$}\\
C^kk^d\eps^{\delta_k} e^{-(1-\frac{l}{d+1})\theta^{-1}\rho_0t_m} & \text{ if $1\leq l\leq d$.}
\end{cases}
\end{equation}

\subsubsection*{Claim.} After passing to a subsequence of $t_m$, there exists $s_m\in\R$ such that $s_m/t_m\to0$, and if we put $t_m''=t_m+s_m$, then 
\begin{equation}
\label{eq:pifix2}
    \lim_{m\to\infty} \norm{\pifix(b_{t_m''}gv_m)}=0. 
\end{equation}

To prove this claim, without loss of generality, after passing to a subsequence we may assume that 
\begin{align}
   \text{either } &\inf_{m\in\N} \norm{\pi_0(b_{t_m}gv_m)}>0,  
    \label{eq:p0-case}\\
   \text{or } &\inf_{m\in\N} \norm{\pi_{d+1}(b_{t_m}gv_m)}>0,
   \label{eq:pd+1-case}
\end{align}
indeed if each infimum is $0$, then choose $s_m=0$. 

First suppose that \eqref{eq:p0-case} holds. Then, from \eqref{eq:each-l-cases} we obtain that $\inf_{m\in\N} E_m>0$. Therefore, $\lim_{m\to\infty}E_me^{\theta^{-1}\rho_0t_m}=\infty$. Let 
\[
s_m=\min\left\{\sqrt{t_m},\frac{d(n+1)}{2(n+1-k)}\log E_me^{\theta^{-1}\rho_0t_m}\right\}. 
\]
Then using \eqref{eq:bt-action} and \eqref{eq:each-l-cases}, we obtain \eqref{eq:pifix2}. 

Now suppose that \eqref{eq:pd+1-case} holds. Then, from \eqref{eq:each-l-cases} we deduce that 
\[
\inf_{m\in\N} E_m^{-1}e^{-\theta^{-1}\rho_0t_m}>0.
\]
Therefore, $\lim_{m\to\infty} E_m^{-1}=\infty$. Let
\[
s_m=-\min\left\{\frac{2(n+1)}{k}\log{E_m^{-1}},\sqrt{t_m}\right\}.
\]
Then \eqref{eq:pifix2} follows from \eqref{eq:bt-action} and \eqref{eq:each-l-cases}. This completes the proof of the Claim. 

Hence, by \eqref{eq:bt-action} and \eqref{eq:each-l-cases}, we get 
\begin{align}
\norm{\pifix(b_{t_m''}gv_m)}&\to 0 \text{ as $m\to\infty$, and } \label{eq:pifix3}\\
\norm{\pi_l(b_{t_m''}gv_m)}&\leq C^kk^d\eps^{\delta_k}\eps^{\delta_k} e^{(\frac{l}{d+1}-\frac{k}{n+1})s_m}e^{-\frac{d+1-l}{d+1}\theta^{-1}\rho_0t_m} \text{, } \forall 1\leq l\leq d. \label{eq:pi-l2}
\end{align}
 
Therefore,
\[
\varphi_{\eps}^\theta(b_{t_m''}gv_m)=\min_{1\leq l\leq d} \eps^{\frac{d+1}{d+1-l}\delta_k}\norm{\pi_l(b_{t_m''}gv_m)}^{-\frac{d+1}{d+1-l}}\geq Ce^{cs_m}e^{\rho_0t_m},
\]
for some $C>0$ and $c\geq 0$ which are independent of $m$. So, as $s_m/t_m\to 0$, we get
    \[
    \rho_{\eps,\theta}(x)\geq \limsup_{m\to\infty} \frac{\log\varphi_{\eps}^\theta(b_{t_m''}gv_m)}{t_m''}\geq \rho_0.
    \]
This completes the proof in Case~2.  
\end{proof}

The following consequence of the above proof is not used in our arguments, but clarifies whether $\eps$ plays any role in the value of $\rho_{\eps,\theta}(x)$. 

\begin{lem} \label{rem:eps12}
    Let $x\in X$, $\theta>0$, and $0<\eps_1\leq\eps$. 
    Then $\rho_{\eps_1,\theta}(x)\leq \rho_{\eps,\theta}(x)$.  
    
    Moreover, if $\rho_{\eps,\theta}(x)\leq\theta$, then $\rho_{\eps_1,\theta}(x)= \rho_{\eps,\theta}(x)$. 
\end{lem}
   
\begin{proof}
    We write $x=gx_0$ for some $g\in G$. To prove the first inequality, without loss of generality, we may assume that $\rho_{\eps_1,\theta}(x)>0$. Then there exist a sequence $t_m\to\infty$, $k\in \{1,\ldots,n\}$, and integral $k$-vectors $v_m$ for all $m$ such that 
    $\norm{\pifix(b_{t_m}gv_m)}<\eps_1^{\delta_k}$, 
    and
    \begin{align*}
    \rho_{\rho_{\eps_1,\theta}(x)}&=\lim_{m\to\infty}\frac{\log\varphi_{\eps_1}^\theta(b_{t_m}gv_m)}{t_m} \\
    &=\lim_{m\to\infty}\frac{\theta\log \min_{1\leq i\leq d} \eps_1^{\frac{d+1}{d-i+1}\delta_k}\norm{\pi_i(b_{t_m}gv_m)}^{-\frac{d+1}{d-i+1}}}{t_m}\\
    &=\lim_{m\to\infty}\frac{\theta\log \min_{1\leq i\leq d} \norm{\pi_i(b_{t_m}gv_m)}^{-\frac{d+1}{d-i+1}}}{t_m} \text{, as $t_m\to\infty,$}\\
    &=\lim_{m\to\infty}\frac{\theta\log \min_{1\leq i\leq d} \eps^{\frac{d+1}{d-i+1}\delta_k}\norm{\pi_i(b_{t_m}gv_m)}^{-\frac{d+1}{d-i+1}}}{t_m}\\
    &= \lim_{m\to\infty} \frac{\log \varphi_\eps^\theta(b_{t_m}gv_m)}{t_m}\leq \rho_{\eps,\theta}(x),
    \end{align*}
    where the last equality holds because $\norm{\pifix(b_{t_m}gv_m)}<\eps_1^{\delta_k}<\eps^{\delta_k}$. 
    
To prove the inequality in the opposite direction,  we further assume that $0<\rho_{\eps,\theta}(x)<\theta$. In the proof of \Cref{prop:bt-Omega}, in \eqref{eq:pifix1} for Case~1 and \eqref{eq:pifix3} for Case~2, we showed that the $\limsup$ in the definition of $\rho_{\eps,\theta}(gx_0)$ can be achieved along a sequence $t_m''\to\infty$ such that $\norm{\pifix(b_{t_m''}gv_m)}\to 0$, and
    \begin{align*}
    \rho_{\eps,\theta}(gx_0)&=\lim_{m\to\infty}\frac{\theta\log \min_{1\leq i\leq d} \eps^{\frac{d+1}{d-i+1}\delta_k}\norm{\pi_i(b_{t_m''}gv_m)}^{-\frac{d+1}{d-i+1}}}{t_m''}\\
    &=\lim_{m\to\infty}\frac{\theta\log \min_{1\leq i\leq d} \norm{\pi_i(b_{t_m''}gv_m)}^{-\frac{d+1}{d-i+1}}}{t_m''}\\
    &=\lim_{m\to\infty}\frac{\theta\log \min_{1\leq i\leq d} \norm{\pi_i(b_{t''_m}gv_m)}^{-\frac{d+1}{d-i+1}}}{t''_m} \text{, as $t_m''\to\infty$,}\\
    &=\lim_{m\to\infty} \frac{\theta\log\varphi_{\eps_1}(b_{t_m''}gv_m)}{t_m''}\leq \rho_{\eps_1,\theta}(x),
    \end{align*}
    where the last equality holds because $\norm{\pifix(b_{t_m''}gv_m)}<\eps_1^{\delta_k}$ for all large $m$. 
    
\end{proof}

\subsection{Properties of the height function and non-divergence}

We recall the following lemma from \cite[Lemma 3.4]{SY24TAMS}. Let $I$ denote the interval $[-1/2,1/2]$. 
    
	\begin{lem}\label{lem:contraction_alpha}
		Let $0<\theta<\frac{d}{d+1}$ and $\eps>0$. There exists $C\geq 1$ such that for any $t\geq0$ there exists $b\geq 0$ such that for any $x\in X$ we have
		\begin{equation}\label{eq:contraction hypothesis alphaEpsTheta}
		\int_{I^d}\alpha_\eps^{\theta}(c_tu(s)x)\de s\leq
		Ce^{-\theta t}\alpha_\eps^{\theta}(x)+b.
		\end{equation}
	\end{lem}

    We also need the log Lipschitz property.
    \begin{lem}\label{lem:log Lipschitz}
        Let $\theta>0$. There exists $C_1\geq1$ such that for every $s\in I^d$ and $t\geq0$,
        \begin{equation} \label{eq:Lipschitz alpha}
            C_1^{-1}e^{-\theta t}\alpha_\eps^\theta(x)\leq\alpha_\eps^\theta(c_tu(s)x) \leq C_1e^{d\theta t}\alpha_\eps^\theta(x).
        \end{equation}
    \end{lem}
    \begin{proof}
        By \eqref{eq:alpha_eps^theta}, it suffices to prove that for every $k$ and every $v\in\bigwedge^k\R^{n+1}$,
        \begin{equation}
            e^{-\theta t}\varphi_\eps^\theta(v)\ll\varphi_\eps^\theta(c_tu(s)v) \ll e^{d\theta t}\varphi_\eps^\theta(v),
        \end{equation}
        where the implied constants are absolute.
        By \eqref{eq:def_phi_eps}, since $\pifix(c_tu(s)v)=\pifix(v)$, it remains to check that for every $1\leq i\leq d$,
        \begin{equation} \label{eq:Lipschitz for vector}
            e^{-t}\norm{\pi_i(v)}^{-\frac{d+1}{d-i+1}}\ll\norm{\pi_i(c_tu(s)v)}^{-\frac{d+1}{d-i+1}} \ll e^{dt}\norm{\pi_i(v)}^{-\frac{d+1}{d-i+1}},
        \end{equation}
        where the implied constants are absolute. 

        We now check \eqref{eq:Lipschitz for vector}. For each $1\leq i\leq d$, since $\pi_i$ is $H_d$-equivariant, in view of \eqref{eq:ct} and \eqref{eq:pi-i}, we have
        \begin{equation} \label{eq:Lipschitz for vector 1}
            e^{-\frac{i}{d+1}t}\norm{\pi_i(v)}\ll\norm{\pi_i(c_tu(s)v)} \ll e^{\frac{d-(i-1)}{d+1}t}\norm{\pi_i(v)}.
        \end{equation}
        Taking the $(-\frac{d+1}{d-i+1})$-th power, we get
        \begin{equation} \label{eq:Lipschitz for vector 2}
            e^{-t}\norm{\pi_i(v)}^{-\frac{d+1}{d-i+1}}\ll\norm{\pi_i(c_tu(s)v)}^{-\frac{d+1}{d-i+1}} \ll e^{\frac{i}{d-i+1}t}\norm{\pi_i(v)}^{-\frac{d+1}{d-i+1}}.
        \end{equation}
        Since $\frac{i}{d-i+1}\leq d$ for every $1\leq i\leq d$, \eqref{eq:Lipschitz for vector} follows from \eqref{eq:Lipschitz for vector 2}.
    \end{proof}

  Fix
    \begin{equation} \label{eq:theta}
    0<\delta<\theta<\frac{d}{d+1}. 
    \end{equation}
    Let $t_0\geq 0$. For any $x\in X$, we define
	\begin{equation} \label{eq:definition_alpha_tilde}
	\at (x)=\int_{t_0}^{\infty}e^{-\delta t}\alpha_\eps^\theta(b_{t}x)\de t.
	\end{equation}
In the rest of the section, we write $\alpha=\alpha_\eps^\theta$ to simplify notation.
	
	\begin{lem} \label{lem:properties_of_alpha'}
		\begin{enumerate}
			\item \label{itm:property1_of_alpha'} $\at \colon X\to[1,+\infty]$ is a lower semicontinuous function. 
           
			\item \label{itm:contraction-alpha}
            There exists $C\geq1$ such that for every $x\in X$ and every $t_1\geq 0$, there exists $b'(t_1)>0$ such that
			\begin{equation} \label{eq:contraction_alpha'}
			\int_{I^d}\at (a_{t_1}u(s)x)\de s\leq Ce^{-(\theta-\delta)t_1}\at (x)+b'(t_1).
			\end{equation}
			\item \label{itm:alpha-Lipschitz}
            $\at $ is Lipschitz with respect to the action of the semigroup $\{b_t\}_{t\geq0}H_d$. More precisely, for every compact subset $\mathcal{O}$ of $\{b_t\}_{t\geq0}H_d$, there exists $C\geq1$ such that
			\begin{equation*}
			\at (gx)\leq C\at (x),\quad\forall g\in\mathcal{O},\forall x\in X.
			\end{equation*}
			\item \label{itm:alpha-proper} 
            For any $M>0$, $\at ^{-1}([0,M])$ is compact.
             \item \label{itm:alpha-finite} 
             Suppose that  $\rho_{\eps,\theta}(x)<\delta$ and 
             \begin{equation}
             \label{eq:t0}
             \sup_{t\geq t_0} \frac{\log\alpha_{\eps,\theta}(b_tx)}{t}<\infty.
             \end{equation} Then $\tilde\alpha(x)<\infty$.
		\end{enumerate}
	\end{lem}

	\subsubsection*{Proof of (\ref{itm:property1_of_alpha'})}
			Suppose there is a sequence $(x_i)_{i\in\N}$ in $X$ such that $x_i\to x$ and $\at (x_i)\leq M$ for each $i$. Since $\alpha$ is lower semicontinuous (see e.g. \cite[Lemma 4.1]{Shi20}), we have $\liminf_{i}\alpha(b_tx_i)\geq\alpha(b_tx)$ for every $t\geq t_0$. Hence by Fatou's lemma, we have
			\[
			\begin{split}
			\at (x)&=\int_{t_0}^{\infty}e^{-\delta t}\alpha(b_{t}x)\de t\\
			&\leq \int_{t_0}^\infty\liminf_i e^{-\delta t}\alpha(b_tx_i)\de t\\
			&\leq \liminf_i\int_{t_0}^\infty e^{-\delta t}\alpha(b_tx_i)\de t\\
			&=\liminf_i\at (x_i).
			\end{split}
			\]
			This proves that $\at $ is lower semicontinuous.

            \subsubsection*{Proof of (\ref{itm:contraction-alpha})}
			 Let $b'=\delta^{-1}b$. We have
			\begin{equation*}
			\begin{split}
			\int_{I^d}\at (a_{t_1}u(s)x)\,\de s &=\int_{I^d}\at(c_{t_1}u(s)b_{t_1}x)\,\de s \\
			&=\int_{I^d}\int_{t_0}^{\infty}e^{-\delta t}\alpha(b_{t}c_{t_1}u(s)b_{t_1}x)\,\de t\,\de s \\
			&=\int_{t_0}^{\infty}e^{-\delta t}\left(\int_{I^d}\alpha(c_{t_1}u(s)b_{t+t_1}x)\,\de s\right)\de t.
\end{split}
			\end{equation*}
            Therefore, by \eqref{eq:contraction hypothesis alphaEpsTheta}, 
            \begin{equation*}
			\begin{split}
            \int_{I^d}\at (a_{t_1}u(s)x)\,\de s 
            &\leq \int_{t_0}^{\infty}e^{-\delta t}(Ce^{-\theta t_1}\alpha(b_{t+t_1}x)+b)\,\de t\\
			&=Ce^{-\theta t_1}\int_{t_0+t_1}^\infty e^{\delta t_1}e^{-\delta t'}\alpha(b_{t'}x)\,
			\de t'+b',\quad t'=t+t_1\\
			&\leq Ce^{-(\theta-\delta)t_1}\int_{t_0}^\infty e^{-\delta t'}\alpha(b_{t'}x)\,\de t'+b'\\
			&=Ce^{-(\theta-\delta)t_1}\at (x)+b'.
			\end{split}
			\end{equation*}
			
		\subsubsection*{Proof of (\ref{itm:alpha-Lipschitz})} By \cite[Lemma 4.1]{Shi20}, for any compact subset $\mathcal{O}$ of $H_d$, there exists $C_1\geq1$ such that $\alpha(hx)\leq C_1\alpha(x)$ for every $h\in\mathcal{O}$ and $x\in X$. Since $b_t$ commutes with $H_d$, we have
			\[
			\begin{split}
			\at(hx) &=\int_{t_0}^{\infty}e^{-\delta t}\alpha(b_{t}hx)\de t\\
			&=\int_{t_0}^{\infty}e^{-\delta t}\alpha(hb_{t}x)\de t\\
			&\leq C_1\int_{t_0}^{\infty}e^{-\delta t}\alpha(b_{t}x)\de t\\
			&=C_1\at(x).
			\end{split}
			\]
			Now it suffices to show that for every $T>0$ there exists $C_2\geq1$ such that for every $0\leq t_1 \leq T$ we have $\at(b_{t_1}x)\leq C_2\at(x)$. Indeed, 
			\[
			\begin{split}
			\at(b_{t_1}x) &=\int_{t_0}^{\infty}e^{-\delta t}\alpha(b_{t}b_{t_1}x)\de t\\
			&=e^{\delta t_1}\int_{t_0+t_1}^\infty e^{-\delta t'}\alpha(b_{t'}x)\de t'\\
			&\leq e^{\delta T}\at(x).
			\end{split}
			\]
			\subsubsection*{Proof of (\ref{itm:alpha-proper})} Since $\at$ is lower semi-continuous by (\ref{itm:property1_of_alpha'}),  $\at^{-1}([0,M])$ is a closed subset of $X$. It follows from the construction of $\at$ that 
            \[
            \at(y)\geq \int_{t_0}^{t_0+1}e^{-\delta t}\alpha(b_ty)\de t \geq e^{-\delta}\min_{0\leq t\leq 1}\alpha(b_ty), \,\forall  y\in X.
            \]
            So, $\at^{-1}([0,M])\subset b([-1,0])\alpha^{-1}([1,e^{\delta}M])$. Since $\alpha^{-1}([1,e^{\delta}M])$ is compact in $X$, it follows that $\at^{-1}([0,M])$ is pre-compact in $X$.
            
            \subsubsection*{Proof of (\ref{itm:alpha-finite})} By the definition \eqref{eq:definition of rho} of $\rho\defeq\rho_{\eps,\theta}(x)$, for every $0<\eps'<\delta-\rho$, there exists $T\geq t_0$ such that for every $t\geq T$, we have $\alpha(b_tx)\leq e^{(\rho+\eps')t}$. By \eqref{eq:t0}, $\sup_{t\in[t_0,T]}\alpha_\eps^\theta(b_tx)<\infty$. Therefore,
            \[
            \at(x) = \int_{t_0}^{\infty}e^{-\delta t}\alpha(b_{t}x)\de t \leq \int_{t_0}^Te^{-\delta t}\alpha(b_tx)\de t+\int_T^\infty e^{(-\delta+\rho+\eps')t}\de t <\infty.
            \]

            This completes the proof of all parts of \Cref{lem:properties_of_alpha'}. \qed

\subsubsection{Obtaining non-divergence}
	Let $x\in X$. For any measurable subset $K$ of $X$ and $T>0$, let
	\begin{equation*}
	\mathcal{C}_K^T(s)= \frac{1}{T}\int_{0}^T\mathbf{1}_K(a_tu(s)x)\de t,
	\end{equation*}
	where $\mathbf{1}_K$ is the characteristic function of $K$.
	
	The following result is the analogue of \cite[Proposition 2.3]{Shi20}.
	
	\begin{prop} \label{prop:continuous_exponential_recurrence}
		Let $\theta<\frac{d}{d+1}$. Suppose that $\rho_{\eps,\theta}(x)<\theta$. Then for every $0<\epsilon<1$, there exist a compact subset $K$ of $X$, $0<a<1$ and $C\geq 1$ such that
		\begin{equation} \label{eq:continuous_exponential_recurrence}
		\abs{ \{ s\in I^d \colon \mathcal{C}_K^T(s)\leq 1-\epsilon \} } \leq Ca^T, \,\forall\, T>0.
		\end{equation}
	\end{prop}
	
	\begin{proof}
    In \eqref{eq:theta} we pick $\delta$ such that $\rho_{\eps,\theta}(x)<\delta<\theta$. Moreover, we can pick $t_0>0$ such that \eqref{eq:t0} holds. Hence $\tilde\alpha(x)<\infty$. 
    
		Note that Lemma 4.9, Lemma 4.10, Corollary 4.11, Lemma 4.12 and Lemma 4.13 in \cite{Shi20} still hold if we replace the function $\alpha$ there with the function $\at$ defined in \eqref{eq:definition_alpha_tilde}, thanks to \Cref{lem:properties_of_alpha'}. Hence the proposition follows from \cite[Lemma 4.12]{Shi20} and \cite[Lemma 4.13]{Shi20} with the above modification.
	\end{proof}
	
	\begin{proof}[Proof of \Cref{thm:main_nondivergence}]
		In view of \eqref{eq:basic identity}, it suffices to prove that for almost every $s\in I^d$, every weak-$\ast$ limit of $\{ \frac{1}{T}\int_{0}^{T}a_tu(s)\delta_{x_A}\de t \}_{T\to \infty}$ is a probability measure on $X$. Indeed, we can cover $\R^d$ by boxes of side length $1$ centered at integral points. The box centered at $\bfp\in\Z^d$ is parametrized by $\begin{psmallmatrix}
		    1&\bfp\\&1
		\end{psmallmatrix}A$, which has the same Diophantine exponent as $A$. 
		
		Let $0<\theta<\frac{d}{d+1}$. Since the Diophantine exponent $\omega(A)<n$, by \cite[Proposition 4.1]{SY24TAMS} we  have that $\rho_{\eps,\theta}(x_A)<\theta$. Let $0<\epsilon<1$ be given. By \Cref{prop:continuous_exponential_recurrence} there exist a compact subset $K$ of $X$, $a<1$, and $C\geq1$ such that \eqref{eq:continuous_exponential_recurrence} holds for all $T\in\N$. Hence by Borel-Cantelli lemma, we have $\liminf_{T\to\infty \text{ in } \N}\mathcal{C}_K^T(s)\geq 1-\epsilon$ holds for almost every $s\in I^d$. Consider the compact set $K_1=\{a_tK:t\in [0,1]\}$. Then, for almost every $s\in I^d$, any weak-$\ast$ limit $\nu$ of $\{\frac{1}{T}\int_0^Ta_tu(s)\delta_{x_A}\,dt\}_{T\to\infty}$ satisfies $\nu(K_1)\geq 1-\epsilon$.
	\end{proof}

    \begin{remark} \label{rem:alternate}
        Alternatively, using \Cref{lem:properties_of_alpha'}, one can obtain an analogue of \Cref{prop:continuous_exponential_recurrence} as done in \cite[Theorem~1.5]{KKLM17} or \cite[Proposition~4.8]{Kha20}, and then deduce \Cref{thm:main_nondivergence} as noted in \cite[Remark~2.1]{KKLM17} or \cite[Theorem~4.3(b)]{Kha20}.
    \end{remark}

	\section{Unipotent invariance and avoidance of singular submanifolds}
	\label{sec:Sing}
	\subsection{Homogeneity of ergodic invariant measures} \label{subsect: avoidance of singular sets}
	By \Cref{thm:main_nondivergence} and \cite[Proposition 2.2]{Shi20}, for almost every $s\in\R^d$, any weak-$\ast$ limit point, say $\nu_s$, of the family of probability measures 
    $\{\frac{1}{T}\int_{0}^{T}a_tu(s)\delta_{x_A}\de t\}_{T\to\infty}$ 
    is a probability measure on $X$ invariant under the unipotent subgroup $U=U_a^+\cap H_d$. We note that for the conclusion of \cite[Proposition 2.2]{Shi20} to hold, it is sufficient that $U\subset U_a^+$, but the $U$-fixed point subspace need not be {\it $a_1$-expanding\/} for finite-dimensional representations of $G$. Moreover, by construction, $\nu_s$ is invariant under the action of $S=\{a_t\}_{t\in\R}$. Let $F=SU$. 
	
	Using Ratner's description of ergodic invariant measures for unipotent flows~\cite{Ratner:measure}, in \cite{EinsiedlerShi19} it is shown that every ergodic invariant measure for the $F$-action on $X$ is homogeneous; here it is important that $S$ is simple, meaning $a_t$ has only one eigenvalue $>1$ and all others are $<1$ for all $t>0$. In fact, the arguments in \cite[Proof of Theorem~1.1, p.427]{EinsiedlerShi19} show that each $F$-invariant ergodic probability measure is concentrated on a closed orbit $gLx_0$ and it is invariant under $gHg^{-1}$, where $H=\bfH(\R)^0$, $L=N_G(H)^0=Z_G(H)^0H$, $\bfH$ is a semisimple algebraic subgroup of $\bfG$ defined over $\Q$ such that $Z_\bfG(\bfH)$ is a $\Q$-anisotropic torus, and $g\in G$ is such that $F\subset gLg^{-1}$. In particular, $Z_G(H)x_0$ is a compact orbit of a $\C$-diagonalizable subgroup of $G$. 
    
    For any proper closed subgroup $L$ of $G$ such that the orbit $Lx_0$ is closed, we define
	\begin{equation} \label{eq:NFL}
	N(F,L)\defeq \{g\in G\colon g^{-1}Fg\subset L\},
	\end{equation}
which is an algebraic variety, and we call the set $N(F,L)x_0$ a singular submanifold. We observe that 
\begin{equation} \label{eq:NFL-nor}
	N(F,L)=N_G(F)N(F,L)N_G(L) \text{ and } N(F,L)=N(F,L^0),
	\end{equation}
    the second equality holds because $F$ is connected. 
    
Being defined over $\Q$, there are only countably many possibilities for $\bfH$. Due to the $F$-ergodic decomposition of $\nu_s$, if $\nu_s$ is not $G$-invariant, then there exists a proper subgroup $\bfH$ of $\bfG$ as above such that 
$
\nu_s(N(F,N_G(H))x_0)>0
$.
Therefore, to prove \Cref{thm:main_equidistribution}, in view of \eqref{eq:equidistribution}, it is enough to prove the following.

\begin{prop}  \label{prop:goal}
Suppose that possibility~\ref{item:algebraic obstruction} of \Cref{thm:main_equidistribution} does not hold. Then for any proper semisimple $\bfQ$-subgroup $\bfH$ of $\bfG$ such that $Z_\bfG(\bfH)$ is a $\bfQ$-anisotropic torus, we have $\nu_s(N(F,N_G(H))x_0)=0$ for almost all $s$, where $\nu_s$ denotes any weak-$\ast$ limit point of $\{\frac{1}{T}\int_{0}^{T}a_tu(s)\delta_{x_A}\de t\}_{T\to\infty}$ in $\cP(X)$. 
\end{prop}

The following sections are devoted to the proof of this proposition. 

\subsection{Larger singular submanifolds with simpler structure} 
\label{subsec:SingSet}
Let $\bfH$ be a proper semisimple subgroup of $\bfG$ defined over $\bfQ$ such that $Z_\bfG(\bfH)$ is $\bfQ$-anisotropic, and $N_G(H)$ contains a conjugate of $F=SU$; that is, $N(F,N_G(H))\neq \emptyset$, see~\eqref{eq:NFL}, where $H=\bfH(\R)^0$. Consider the representation of $\bfH$ obtained from the restriction of the standard representation of $\bfG=\bf{SL}_{n+1}=\SL(\bfV)$, where $\bfV$ is an $(n+1)$-dimensional vector space over $\Q$. 
Then, the $\bfH$-action on $\bfV$ has only one $\Q$-isotypic component. Now, using \cite[Theorem 7.2]{Tits71} and the fact that a conjugation of the simple one-parameter $\R$-diagonalizable subgroup $S=\{a_t\}$ is contained in $Z_G(H)H$, in \cite[Proof of Proposition~2.1]{EinsiedlerShi19} it is shown that there exists a number field $\bbF$ contained in $\R$, a vector space $\bfW$ over $\bbF$, and a representation $\rho:\bfH\to\SL(\bfW)$ defined over $\bbF$ which is irreducible over $\C$ such that 
    \[
    \bfV\cong \Res_{\bbF/\bbQ} \bfW.
    \]
    In particular, $\bfH$ acts irreducibly over $\bbQ$ and the centralizer 
    $Z_{\bfG}(\bf H)$ is abelian, and hence $Z_\bfG(\bfH)$ is a $\Q$-anisotropic torus. 

    \begin{prop} \label{prop:sln} 
        Suppose $L$ acts irreducibly on $\bfV$ over $\C$ and $L$ contains a conjugate of $F=SU$. Then $L=\SL_{n+1}(\R)$. In particular, $\bf H={\bf SL}_{n+1}$.
    \end{prop}

    \begin{proof}
        Since $L$ contains a conjugate of $U$ and $L$ is reductive, by a theorem due to Jacobson--Morozov and Kostant, a conjugate of $H_2\cong\SL_2(\R)$ is contained in $L$. Therefore by \cite[Theorem A.7]{SY24PLMS}, either $L=\SL_{n+1}(\R)$ or $n+1$ is even and $L={\Sp}(\R^{n+1},\omega)$ for a symplectic form $\omega$ on $\R^{n+1}$. Now eigenvalues of a symplectic matrix occur in reciprocal pairs, so a conjugate of $S=\{a_t\}_{t\in\R}$ is not contained in ${\Sp}(\R^{n+1},\omega)$ for $n>1$. Since a conjugate of $S$ is contained in $L$, we conclude that $L= \SL_{n+1}(\R)$; we note that $\Sp(\R^2,\omega)=\SL_2(\R)$.
    \end{proof}
    
 We have that $\bfH\neq \bfG$, so $N_G(H)\neq G$. And $N_G(H)$ contains a conjugate of $F$, So, by Proposition~\ref{prop:sln}, $N_G(H)$, and hence $H$, does not act irreducibly on $V$ over $\C$. Therefore, $\bbF\neq \Q$. 
 
    Let $m=[\bbF:\Q]$. Now $m=s_1+2s_2$, where $\{\sigma_i:1\leq i\leq s_1\}$ are the real and 
    $\{\sigma_{s_1+i},\overline{\sigma_{s_1+i}}:1\leq i\leq s_2\}$ are the complex embeddings of $\mathbb{F}$. Let $\{\alpha_1,\ldots,\alpha_m\}$ be a basis of $\bbF$ over $\bbQ$. Now,
    \begin{equation} \label{eq:V-Vi}
    V=\prod_{i=1}^{s_1+s_2} V_i \text{, where } 
    V_i=\begin{cases}
        \bfW^{\sigma_i}({\bbR}) \text{, if } 1\leq i\leq s_1\\
        \bfW^{\sigma_i}(\C)
        \text{, if } s_1+1\leq i\leq s_1+s_2 \text{, as a $\R$-vector space,}
    \end{cases}
    \end{equation}
    and 
    \begin{equation} \label{eq:H-Hi}
    H=\prod_{i=1}^{s_1+s_2}{H_i} \text{, where } 
    H_i=\begin{cases}
        \rho^{\sigma_i}(\bfH)(\bbR) \text{ for } 1\leq i\leq s_1\\
        \rho^{\sigma_i}(\bfH)(\C)\text{ for } s_1+1\leq i\leq s_1+s_2 \text{, as a real Lie group,}
    \end{cases}
    \end{equation}
    and $H_i$ acts irreducibly on $V_i$ for each $i$; see \cite[p.~424, (2.2)]{EinsiedlerShi19}. 
    Now
    \begin{equation} \label{eq:Z-Zi}
    Z_G(H)=G\cap \prod_{i=1}^{s_1+s_2} Z_i \text{, where } 
    Z_i=\begin{cases}
        \R^\ast I_{V_i} \text{ for } 1\leq i\leq s_1 \\
        \C^\ast I_{V_i} \text{ for } s_1+1 \leq i\leq s_1+s_2.
    \end{cases}
    \end{equation}
    
    Since $\bbF\subset \R$, $s_1\geq 1$ and since $\bbF\neq\bbQ$, we have $m\geq 2$. Therefore, $\R$-rank of $Z_G(H)$ is $s_1+s_2-1\geq 1$. 
    Let $\bfM$ be the centralizer of $Z_G(\bf H)$ in $\bfG$. Then
    \begin{equation} \label{eq:M-prod}
       M=G\cap\prod_{i=1}^{s_1+s_2}J_i \text{, where } 
    J_i=\begin{cases}
        \GL_\R(\bfW^{\sigma_i}({\bbR})) \text{ for } 1\leq i\leq s_1\\
        \GL_\C(\bfW^{\sigma_i}(\C))\text{ for } s_1+1\leq i\leq s_1+s_2.
    \end{cases} 
    \end{equation}
    
    Let $r=\dim_\bbF \bfW$. Then, $\dim V_i=r$ for $1\leq i\leq s_1$ and $\dim V_{i}=2r$ for $s_1+1\leq i\leq s_1+s_2$. So, $n+1=rm$. 
    
    Now $\bfM$ is a reductive $\Q$-subgroup of $\bfG$ and its center equals $Z_\bfG(\bfH)$.  
    Since $Z_\bfG(\bfH)$ is $\Q$-anisotropic, $\bfM$ admits no nontrivial $\Q$-characters. So, $Mx_0$ admits a finite $M$-invariant measure, and in particular, $Mx_0$ is closed in $X$. Also, $N_G(H)^0=HZ_G(H)^0\subset M$. So, by \eqref{eq:NFL-nor},
    \begin{equation}
        \label{eq:NFM}
         N(F,M)\supset N(F,N_G(H)^0)=N(F,N_G(H))\neq \emptyset. 
    \end{equation}
   
   \subsubsection{Claim: $r\geq d+1$} \label{claim:3.2.1} As we will explain, this is a consequence of the fact that the dimension of the strictly expanding weight space of $\Ad(a_t)$ on the Lie algebra of $F$ is $d$. Due to \eqref{eq:NFM}, pick $g\in G$ be such that $gFg^{-1}\subset M$; so $ga_tg^{-1}$ preserves $V_i$ for each $i$. Therefore, the eigen space of $ga_tg^{-1}$ corresponding to the simple eigenvalue $e^{\frac{n}{n+1}t}$ is contained in exactly one $V_i$ corresponding to a real embedding of $\bbF$, as any complex embedding would have repeated real eigenvalues; and hence on each of the $V_j$ for $j\neq i$, $ga_tg^{-1}$ acts as a scalar. Therefore, the strictly expanding eigen space of $\Ad(ga_tg^{-1})$ on the Lie algebra of $M$ is contained in the Lie algebra of $\GL(V_i)$ and hence its dimension is $r-1$. Since $M$ contains $gFg^{-1}$, and the dimension of the expanding eigen space of $\Ad(ga_tg^{-1})$ on the Lie algebra of $gFg^{-1}$ is $d$, we get $r-1\geq d$. This proves the claim. 

   \subsubsection{Choice of a conjugate of $M$} \label{sec:v0}
     Let $v_0$ be a regular $\Q$-point of the Lie algebra of $Z_G(H)$. Then $\bfM$ is the isotropy group of $v_0$ for the adjoint representation of $\bfG$. Let $g_1\in G$ be a change of basis matrix corresponding to the decomposition~\ref{eq:V-Vi}, where for $i\geq s_1+s_2$, the multiplication by $z\in\C$ is given by the matrix $\varphi(z)=\begin{psmallmatrix}
        \Re(z)I_r & \Im(z)I_r \\
        -\Im(z)I_r & \Re(z)I_r
    \end{psmallmatrix}$. Due to the construction of restriction of scalars (see, e.g.~\cite[Proposition~6.1.3]{Zimmer1984ErgodicGroups}), there exists $x\in \bbF$ such that $\Ad g_1(v_0)$ is of the form
	\begin{equation} \label{eq:adg1v0}
	\begin{psmallmatrix}
	\sigma_{1}(x)I_r & & & & & \\
	& \ddots & & & & \\
	& &  \sigma_{s_1}(x)I_r & & & \\
	& & & \varphi(\sigma_{s_1+1}(x))& & \\
	& & & & \ddots & \\
	& & & & & \varphi(\sigma_{s_1+s_2}(x))
	\end{psmallmatrix}.
	\end{equation}
    Due to \eqref{eq:V-Vi} and \eqref{eq:M-prod}, the $\sigma_i(x)$ are all distinct for $1\leq i\leq s_1+s_2$. In particular, $x\not\in\Q$. 

    \begin{remark} \label{rem:g1-F}
    We have chosen $g_1$ such that for $1\leq i\leq s_1$, the $i$-th block of $(r\times r)$-rows of $g_1$ are defined over $\sigma_i(\bbF)$. 
\end{remark}

    We set $M_1=g_1Mg_1^{-1}$. Then $M_1$ is the centralizer of $\Ad(g_1)v_0$.  As a consequence of Claim~\ref{claim:3.2.1}, $H_d\subset M_1$. Moreover the full diagonal subgroup of $G$ is contained in $M_1$. In particular, $F\subset M_1$, so $g_1\in N(F,M)$. 

 \begin{prop}
        \label{prop:ZFM}
        There exist $p_1,\ldots,p_{s_1}\in N_G(M_1)$ such that 
        \[
        N(F,M)=\bigcup_{i=1}^{s_1} Z_G(F)M_1p_ig_1.
        \]
        In fact, for each $1\leq i\leq s_1$, we let $p_i\in \bfG(\Z)$ to be a permutation on coordinates such that under conjugation it sends the $i$-th $(r\times r)$-block of $M_1$ to the first $(r\times r)$-block of $M_1$. 
    \end{prop}
    \begin{proof} Let $g\in N(F,M)$. Then $g^{-1}Fg\subset M$. So $(g_1g^{-1})F(gg_1^{-1})\subset M_1$. Now the eigenvalue $e^{\frac{n}{n+1}t}$ of $(g_1g^{-1})a_t(gg_1^{-1})$ must occur in the $i$-th $(r\times r)$-blocks of $M_1$ where $1\leq i\leq s_1$. Let $p=p_i\in N_G(M_1)$ be as chosen in the statement of the proposition. Therefore, the eigenvalue of $(pg_1g^{-1})a_t(pg_1g^{-1})^{-1}$ is in the first $r\times r$-block of $M_1$. Further, we can find $m_1\in M_1$ such that $(m_1pg_1g^{-1})a_t(m_1pg_1g^{-1})^{-1}=a_t$. Moreover, 
    \[
(m_1pg_1g^{-1})U(m_1pg_1g^{-1})^{-1}\subset m_1p(g_1g^{-1})F(gg_1^{-1})p^{-1}m_1^{-1}\subset M_1.
    \]
    Now, $(m_1pg_1g^{-1})U(m_1pg_1g^{-1})^{-1}$ is contained in the expanding horospherical subgroup of $a_t$ in $M_1$, so it is contained in the first row of the first $r\times r$-block of $M_1$. So we can pick, $m_2$ in the first block of $M_1$ such that $m_2$ commutes with $a_t$ and 
\[m_2(m_1pg_1g^{-1})u(m_1pg_1g^{-1})^{-1}m_2^{-1}=u, \, \forall u\in U.
\]
Let $z=(m_2m_1pg_1g^{-1})^{-1}$. Then $z\in Z_G(F)$ and $g=zm_2m_1pg_1\in zM_1pg_1$. There are only $s_1$ possibilities for $p$ as above; they are representatives of $Z_G(F)\backslash N_G(M_1)/M_1$ in $N_G(M_1)$.
    \end{proof}

    \subsection{Construction of a height function with respect to the singular submanifold} \label{sec:Mf}

    Fix any $0<\theta<\frac{d}{d+1}$. We pick $\theta'$ such that 
    \begin{equation} \label{eq:theta and theta'}
        \frac{d}{d+1}\theta'<\theta<\frac{d}{d+1}.
    \end{equation}

  For any $l\geq 1$, we define
    \begin{equation*}
        X_l = \{ x\in X \mid \alpha_\eps^\theta(x)\leq l \}.
    \end{equation*}
    Since $\alpha^\theta_\eps$ is lower semicontinuous and proper, $X_l$ is a compact subset of $X$. Let
    \begin{equation*}
        X'=\{ x\in X \mid \alpha_\eps^\theta(x)<+\infty \}.
    \end{equation*}
    As we have noted in the paragraph following \eqref{eq:alpha_eps^theta}, $H_dX'=X'$.

    \subsubsection{Around a chosen section of the singular submanifold}  
In view of \Cref{prop:ZFM}, pick $p=p_i\in N_G(M_1)$ for some $1\leq i\leq s_1$, and let 
\begin{equation} \label{eq:Y}
\text{$y=pg_1x_0$, $Y=M_1y=pg_1Mx_0$,}
\end{equation}
which is a closed orbit of $M_1$ in $X$.  

Let $\fs$ denote the $\Ad(H_d)$-invariant subspace of $\Lie(Z_G(H_d))$ which is complementary to $\Lie(M_1)$. Then $\fg=\Lie(G)$ can be expressed as a direct sum of $\Ad(H_d)$-invariant subspaces 
    \begin{equation} \label{eq:g=zmpq}
        \fg = \fq \oplus \mathfrak{p}\oplus \fs\oplus\Lie(M_1),
   \end{equation}
    where
    \begin{align}
        \mathfrak{q}=\left\{ 
        \begin{psmallmatrix}
            & B_1 \\
             &
        \end{psmallmatrix}: B_1\in\Mat_{d+1,n+1-r} \right\}, \qquad
        \mathfrak{p}=\left\{ 
        \begin{psmallmatrix}
            & \\
            B_2 &
        \end{psmallmatrix}: B_2\in\Mat_{n+1-r,d+1} \right\}.
        \label{eq:p}
    \end{align}
 
    \begin{remark} \label{rem:Hd on p+q}
        Note that $\mathfrak{q}$, as a $H_d\cong\SL_{d+1}(\R)$-module, is a direct sum of $(n+1-r)$-copies of the standard representation $\R^{d+1}$, and $\mathfrak{p}$ is a direct sum of $(n+1-r)$-copies of the contragradient of the standard representation, which is isomorphic to $\wedge^d \R^{d+1}$.

        In particular, the eigenvalues of $c_t$-action on $\mathfrak{q}$ are $e^{\frac{d}{d+1}t}$ and $ e^{-\frac{1}{d+1}t}$; and on $\mathfrak{p}$ they are $e^{\frac{1}{d+1}t}$ and $e^{-\frac{d}{d+1}t}$. 
    \end{remark}

As a consequence we have the following:

\begin{lem} \label{lem:Lipschitz in q}
        For every $v\in\mathfrak{q}$, $w\in\mathfrak{p}$, $t\geq0$ and $s\in \R^d$,
        \begin{equation} \label{eq:Lipschitz for v}
            e^{-\frac{1}{d+1}t}\norm{u(s)^{-1}}^{-1}\norm{v} \ll \norm{c_tu(s)v} \ll e^{\frac{d}{d+1}t}\norm{u(s)}\norm{v},
        \end{equation}
        \begin{equation} \label{eq:Lipschitz for w}
            e^{-\frac{d}{d+1}t}\norm{u(s)^{-1}}^{-1}\norm{w} \ll \norm{c_tu(s)w} \ll e^{\frac{1}{d+1}t}\norm{u(s)}\norm{w}.
        \end{equation}
    \end{lem}

\begin{remark}
    \label{rem:diffeo}
    For any $r>0$, let $B_r$ denote the norm ball of radius $r$ in $\fg$ and the map
    \[
    \psi:(\fq\cap B_r)\times(\fp\cap B_r)\times (\fs\cap B_r)\times (\Lie(M_1)\cap B_r)\to G,
    \]
    given by $(v,w,z,y)\mapsto \exp(-v)\exp(-w)\exp(z)\exp(y)$, and let $\Psi_r$ denotes its image. Then fix $0<r_0<1$ such that $\psi_{r_0}$ is a diffeomorphism onto $\Psi_{r_0}$.  
\end{remark}

    We will need an estimate for return to a fixed compact set in terms of $\alpha_\eps^\theta(x)$. 

\begin{lem} \label{lem:back to compact}
        Let $\sigma_0>1$ be given. Then there exist $t_1\geq 1$ and $l_1\geq 1$ such that for any $l\geq l_1$, and any $x\in X'$ the following holds: There exist 
        \[
        \tau\leq t_1+\sigma_0\theta^{-1}\log^+ (2l^{-1}\alpha_\eps^\theta(x)) \text{ and } s\in [-1,1]^d
        \]
        such that
        \[
        c_{\tau}u(s)x \in X_l.  
        \]
        Here $\log^+(y):=\max\{ \log y,0 \}$ for every $y>0$.
    \end{lem}
    \begin{proof}
       Let $C>1$ be such that the conclusion of \Cref{lem:contraction_alpha} holds. Take $t_1\geq 1$ sufficiently large such that 
        \[
        \frac{\theta}{\theta-t_1^{-1}\log C}\leq\sigma_0.
        \]
        We will write $t=t_1$ for simplicity. By \Cref{lem:contraction_alpha}, there exists $b>0$ such that 
        \[
        \int_{I^d}\alpha_\eps^{\theta}(c_tu(s)x)\de s\leq
		Ce^{-\theta t}\alpha_\eps^{\theta}(x)+b.
        \]
        Iterating $n$ times,
        \begin{equation} \label{eq:exponential contraction}
        \begin{split}
            \int_{\R^d}\alpha_\eps^{\theta}(c_{nt}u(s)x)\de \nu^{(n)}(s)
            &\leq C^ne^{-n\theta t}\alpha_\eps^{\theta}(x)+\frac{1-C^ne^{-n\theta t}}{1-Ce^{-\theta t}}b \\
            &< C^ne^{-n\theta t}\alpha_\eps^{\theta}(x)+\frac{b}{1-Ce^{-\theta t}},
        \end{split}    
        \end{equation}
        where $\nu^{(n)}$ is a probability measure on $\R^d$ supported on $(1-e^{-t})^{-1}I^d$. Indeed, $\nu^{(1)}$ is the Lebesgue measure restricted to $I^d=[-1/2,1/2]^d$, and for each $n\geq 2$,
        \begin{align*}
&\int_{I^d}\left(\int_{\R^d}\alpha_\eps^\theta(c_{(n-1)t}u(s_1)c_{t}u(s)x)\,d\nu^{(n-1)}(s_1)\right)\,d\nu^{(1)}(s)\\
&=\int_{I^d}\left(\int_{\R^d}\alpha_\eps^\theta(c_{nt}u(e^{-t}s_1+s)x)\,d\nu^{(n-1)}(s_1)\right)\,d\nu^{(1)}(s)\\
&=\int_{\R^d}\alpha_\eps^{\theta}(c_{nt}u(s)x)\,d\nu^{(n)}(s),
\end{align*}
    where $\nu^{(n)}$ is the convolution of $(e^{-t}I_d)_\ast \nu^{(n-1)}$ and $\nu^{(1)}$, where $e_{-t}I_d$ denotes the multiplication by $e^{-t}$ on $\R^d$. 

    \begin{remark} \label{rem:nu*n}
       For later purpose we note that $\de\nu^{(n)}=f_n\de s$, where $f_n\geq 0$, $f_n$ is symmetric about each coordinate axis, $\int f_n=1$, and if $t_1\geq 2$, then $f_n\geq 4^{-d}\mathbf{1}_{I^d}$; that is, $\nu^{(1)}\leq 4^d\nu^{(n)}$ for all $n$.  
    \end{remark}
        
        Take $l_1=\frac{2b}{1-Ce^{-\theta t}}$. Let $l\geq l_1$ and choose smallest $n\in\N$ such that 
        \[
        C^ne^{-n\theta t}\alpha_\eps^{\theta}(x)\leq l/2. 
        \]
        So, we choose
        $n=\left\lceil \frac{\log^+ (2l^{-1}\alpha_\eps^\theta(x))}{\theta t-\log C} \right\rceil$.
        Then, the right hand side of \eqref{eq:exponential contraction} is at most $l/2+l_1/2\leq l$. Hence, by \eqref{eq:exponential contraction}, there exist 
        \[
        s\in {(1-e^{-t})}^{-1}I^d\subset [-1,1]^d
        \]
        such that $\alpha_\eps^{\theta}(c_{nt}u(s)x)\leq l$.  Now, $t=t_1$ and
        \begin{align*}
            nt&\leq t+\frac{t}{\theta t-\log C}\cdot \log^+ (2l^{-1}\alpha_\eps^\theta(x)) \\
            &=t_1+\theta^{-1}\cdot \frac{\theta}{\theta-t_1^{-1}\log C}\cdot \log^+ (2l^{-1}\alpha_\eps^\theta(x))\\
            &\leq t_1+\theta^{-1}\sigma_0\log^+ (2l^{-1}\alpha_\eps^\theta(x)).
        \end{align*}
       We conclude the proof of the lemma by choosing $\tau=nt_1$.
    \end{proof}

    \begin{lem} \label{lem:size of r_x}
         Let $\sigma>1$. We can pick an open neighborhood $O$ of the 0 in $\fs$ and an $0<\eps_0<r_0^d$ such that for any $x\in X$, if we set 
    \begin{equation} \label{eq:r_x}
        r_x\defeq \eps_0\left[\alpha_\eps^{\frac{d}{d+1}\sigma}(x)\right]^{-1},
    \end{equation}
    then for every $x\in X'$, there exists at most one $v\in\mathfrak{q}$ and $w\in\mathfrak{p}$ such that
    \begin{equation}\label{eq:unique_representative}
        \exp(w)\exp(v)x\in \exp(O)Y\quad\text{  and  }\quad\norm{v}<r_x,\norm{w}<(r_x)^{\frac{1}{d}}.
    \end{equation}
    \end{lem}

    \begin{proof}
        This follows from \Cref{lem:Lipschitz in q}, \Cref{lem:back to compact} and \cite[Lemma 6.6]{PSS23}; c.f. the claim before \cite[eq.(6.15)]{PSS23}. We give a detailed proof below.

        Pick $\sigma_0$ such that $1<\sigma_0<\sigma$, and let $t_1$ and $l_1$ as in \Cref{lem:back to compact}. Pick any $l\geq l_1$. Now $X_l$ is compact, and hence the set $\tilde{X_l}=\overline{\Psi_{1}}X_l$ is compact, where $\Psi_r$ is as in \Cref{rem:diffeo}. 

        Since $M_1y$ is closed in $X$, there exists $0<r_1<r_0$ such that for any $g\in \Psi_{r_1}$, if $g(\tilde X_l\cap M_1y)\cap M_1y\neq \emptyset$, then $g\in M_1$. 
        
        Let $0<r<r_0$ be such that  $\Psi_{r}^{-1} \Psi_{r}\cap \Stab(y)=\{e\}$ for all $y\in Y$ and $(\Psi_r^{-1}\Psi_r)\subset \Psi_{r_1}$ and $\Psi_r^{-1}\subset \Psi_1$.

    Let $O=\fs\cap B_r$. 
       Then for any $x'\in X_l$, there exists at most one $v'\in\mathfrak{q}$ and $w'\in\mathfrak{p}$ such that
    \begin{equation}\label{eq:unique_representative1}
        \exp(w')\exp(v')x'\in \exp(O)Y\quad\text{  and  }\quad\norm{v'}<r,\norm{w'}<r.
    \end{equation}
Indeed, suppose $\exp(w_i)\exp(v_i)x'=\exp(z_i)m_iy$ for some $v_i\in\fq\cap B_r$, $w_i\in\fp\cap B_r$, and $z_i\in\fs\cap B_r$ and $m_i\in M_1$ for $i=1,2$. Then $m_iy\in \Psi_r^{-1}x'\subset\Psi_1X_l$. Let 
\[
g=(\exp(-v_2)\exp(-w_2)\exp(z_2))^{-1}\exp(-v_1)\exp(-w_1)\exp(z_1)\in \Psi_{r}^{-1}\Psi_r\subset\Psi_{r_1}.
\]
Then, $g(m_1y)=m_2y$, and hence $g\in M_1$ by our choice of $r_1$ above. So, 
\[
g\in M_1\cap \Psi_{r_1}\subset\exp(\Lie(M_1)\cap B_{r_1}).
\]
Then 
\[
\exp(-v_2)\exp(-w_2)\exp(z_2)g=\exp(-v_1)\exp(-w_1)\exp(z_1).
\]
Therefore, by \Cref{rem:diffeo}, we get $v_1=v_2$ and $w_1=w_2$. This proves there is at most one solution to \eqref{eq:unique_representative1}.
    
    Let $x\in X'$. By \Cref{lem:back to compact}, there exists $\tau\leq t_1+\sigma_0\theta^{-1}\log (2l^{-1}\alpha_\eps^\theta(x))$ and $s\in\R^d$ with $\norm{s}_{\infty}\leq 1$ 
    such that
        \begin{equation} \label{eq:inXl}
        c_{\tau}u(s)x \in X_l.  
        \end{equation}
    Then by \Cref{lem:Lipschitz in q}, for all $v\in\mathfrak{q}$, we have
    \begin{align} 
        \norm{\Ad(c_\tau u(s))v}
        &\leq e^{\frac{d}{d+1}\tau}\cdot \norm{u(s)}\norm{v} 
        \notag\\
        &\leq e^{\frac{d}{d+1}t_1} (2\ell^{-1})^{\frac{d}{d+1}\sigma_0\theta^{-1}} \alpha_\eps^{\frac{d}{d+1}\sigma_0}(x)\cdot \norm{u(s)}\norm{v} \notag\\
        &\ll \alpha_\eps^{\frac{d}{d+1}\sigma}(x)\norm{v},\label{eq:expansion v}
        \end{align}
        and for all $w\in\mathfrak{p}$, we have
        \begin{align}
    \norm{\Ad(c_\tau u(s))w} 
        \leq e^{\frac{1}{d+1}\tau}\norm{u(s)}\norm{w}
        \ll \alpha_\eps^{\frac{1}{d+1}\sigma}(x)\norm{w}.\label{eq:expansion w}
    \end{align}
    
    For $\eps_0>0$, let $r_x$ be as in \eqref{eq:r_x}. Now suppose $v\in\fq$ and $w\in\fp$ satisfy \eqref{eq:unique_representative}. Then, 
    \begin{align} 
   &c_\tau u(s)\exp(w)\exp(v)x \notag\\
   &=\exp(\Ad(c_\tau u(s))w)\exp(\Ad(c_\tau u(s))v)(c_\tau u(s))x\in \exp(O)Y,
   \label{eq:inexpOY}
    \end{align}
    because $c_\tau u(s)\exp(O)Y=\exp(O)Y$. 
    By \eqref{eq:r_x}, \eqref{eq:unique_representative}, \eqref{eq:expansion v} and \eqref{eq:expansion w}, 
    \[
    \norm{\Ad(c_\tau u(s))v}\ll \eps_0<r \text{ and } \norm{\Ad(c_\tau u(s))w}\ll \eps_0^{\frac{1}{d}}<r,
    \]
    where we pick $\eps_0>0$ sufficiently small depending only on the implicit constants in \eqref{eq:expansion v} and \eqref{eq:expansion w} and $r$. Then, by \eqref{eq:unique_representative1}, in view of \eqref{eq:inXl} and \eqref{eq:inexpOY}, we conclude that the $v$ and $w$ are unique.
    \end{proof}

    \subsubsection*{Choices of parameters}
    We have picked $\theta$ and $\theta'$ to satisfy \eqref{eq:theta and theta'}. Now we pick $\sigma>1$ such that
    \begin{equation} \label{eq:choice of sigma}
        \frac{d}{d+1}\theta'<\sigma \frac{d}{d+1}\theta'< \theta <\frac{d}{d+1}.
    \end{equation}
    For this choice of $\sigma>1$, we pick $\sigma_0$ such that $1<\sigma_0<\sigma$, and let $t_1$ and $l_1$ as in \Cref{lem:back to compact}. Further we pick any $l\geq l_1$, which we shall specify later, we obtain $O$ and $r_x$ such that the conclusion of  \Cref{lem:size of r_x} holds. We set 
    \begin{equation}
        \label{eq:cN-def}
        \cN=\exp(O)Y=\exp(O)M_1y.
    \end{equation}

    \subsubsection{Height function for \texorpdfstring{$\cN$}{\unichar{U+1D4A9}} with \texorpdfstring{$c_tu(I^d)$}{c\textsubscript t u(I\textsuperscript d)}-contraction}
     Define $\beta^{\theta'}:X\to[1,\infty]$ by
	\begin{align} \label{eq:defbeta}
	\beta^{\theta'}(x)&=\begin{cases}
	\min\{\norm{v}^{-\theta'},\norm{w}^{-d\theta'}\} &\text{if }\exists!\, (v,w)\in \fq\times\fp \text{ satisfying } \eqref{eq:unique_representative} \\ 
	r_x^{-\theta'} &\text{otherwise.}
	\end{cases}
	\end{align}

   \begin{remark} \label{rem:beta-larger}
For any $x\in X$,
\[
\beta^{\theta'}(x)=\max\bigl(\{r_x^{-\theta'}\}\bigcup\bigl\{\min\{\norm{v}^{-\theta'},\norm{w}^{-d\theta'}\}:v\in \fq, w\in\fp, \exp(w)\exp(v)x\in\cN\bigr\}\bigr). \label{eq:beta-min}
\]

To see this, let $x\in X$, and $(v,w)\in \fq\times\fp$. Then, \eqref{eq:unique_representative} is equivalent to 
\begin{equation} \label{eq:equiv-unique-rep}
\exp(w)\exp(v)x\in \cN \text{ and } r_x^{-\theta'}<\min\{\norm{v}^{-\theta'},\norm{w}^{-d\theta'}\}.
\end{equation}
 Therefore,   
       $r_x^{-\theta'}\leq \beta^{\theta'}(x)$. 
       
       Moreover, if $\exp(w)\exp(v)x\in \cN$, then we claim that
       \begin{equation} \label{eq:beta-geq-norminv}
           \beta^{\theta'}(x)\geq \min\{\norm{v}^{-\theta'},\norm{w}^{-d\theta'}\}. 
       \end{equation}
       To show this by contradiction, suppose \eqref{eq:beta-geq-norminv} fails to hold, then \eqref{eq:equiv-unique-rep} holds. Hence \eqref{eq:unique_representative} hold. In particular, $r_x^{-\theta'}<\infty$, so $x\in X'$. And by \Cref{lem:size of r_x}, there exists at most one pair $(v,w)\in\fq\times\fp$ satisfying \eqref{eq:unique_representative}. So \eqref{eq:beta-geq-norminv} holds with equality, which is a contradiction. 
\end{remark}

    \begin{lem} \label{lem:Lipschitz for beta}
        There exists $C_2\geq1$ such that given any $x\in X'$, $t\geq0$ and $s\in I^d$,
            \begin{equation} \label{eq:Lipschitz beta}
                \beta^{\theta'}(c_tu(s)x)\leq C_2e^{\frac{\sigma d^2\theta'}{d+1}t}\beta^{\theta'}(x).
            \end{equation}
    \end{lem}
    \begin{proof}
        By \Cref{lem:Lipschitz in q}, there exists $C_3\geq1$ such that for all $t\geq 0$, $s\in I^d$, $v\in\mathfrak{q}$ and $w\in\mathfrak{p}$, we have
        \begin{equation} \label{eq:beta Lipschitz}
            C_3^{-1}e^{-\frac{d\theta'}{d+1}t}\norm{v}^{-\theta'}\leq \norm{c_tu(s)v}^{-\theta'} \leq C_3e^{\frac{\theta'}{d+1}t}\norm{v}^{-\theta'},
        \end{equation}
        and
        \begin{equation} \label{eq:beta Lipschitz 2}
            C_3^{-1}e^{-\frac{d\theta'}{d+1}t}\norm{w}^{-d\theta'}\leq \norm{c_tu(s)w}^{-d\theta'} \leq C_3e^{\frac{d^2\theta'}{d+1}t}\norm{w}^{-d\theta'}.
        \end{equation}
        
        We observe that, by \eqref{eq:r_x} and then by \Cref{lem:log Lipschitz}, 
        \begin{equation}\label{eq:r_x Lipschitz}
            r_{c_tu(s)x}^{-\theta'}\ll \left(\alpha_\eps^{\frac{d}{d+1}\sigma}(c_tu(s)x)\right)^{\theta'}
            \ll \left(e^{d\frac{d}{d+1}\sigma t}\alpha_\eps^{\frac{d}{d+1}\sigma}(x)\right)^{\theta'}\ll e^{\frac{\sigma d^2\theta'}{d+1}t}r^{-\theta'}_x.
        \end{equation}
        
       Suppose first that $\beta^{\theta'}(c_tu(s)x)=r^{-\theta'}_{c_tu(s)x}$. By \eqref{eq:r_x Lipschitz} and \Cref{rem:beta-larger}, we have
        \begin{equation} \label{eq:beta case 1}
            \beta^{\theta'}(c_tu(s)x)=r^{-\theta'}_{c_tu(s)x}\ll e^{\frac{\sigma d^2\theta'}{d+1}t}r^{-\theta'}_x\leq e^{\frac{\sigma d^2\theta'}{d+1}t}\beta^{\theta'}(x).
        \end{equation}

        Otherwise, we have that 
        \[
        \beta^{\theta'}(c_tu(s)x)=\min\{\norm{v'}^{-\theta'},\norm{w'}^{-d\theta'}\}>r_{c_tu(s)x}^{-\theta'}
        \]
        for unique $v'\in\mathfrak{q}$ and $w'\in\mathfrak{p}$ such that $\exp(w')\exp(v')c_tu(s)x\in\cN$. Let 
        \[
        v=\Ad(c_tu(s))^{-1}v'\in \mathfrak{q} \text{ and } w=\Ad(c_tu(s))^{-1}w'\in\mathfrak{p}.
        \]
        Then, $c_tu(s)\exp(w)\exp(v)x\in\cN$. Since $H_d\cN=\cN$, we have $\exp(w)\exp(v) x\in \cN$. 
        Then by \eqref{eq:beta Lipschitz},  \eqref{eq:beta Lipschitz 2}, and \Cref{rem:beta-larger}, we have
        \begin{equation} \label{eq:beta case 2}
        \begin{split}
            \beta^{\theta'}(c_tu(s)x) &=\min\{\norm{v'}^{-\theta'},\norm{w'}^{-d\theta'}\} \\
            &\ll \min\{e^{\frac{\theta'}{d+1}t}\norm{v}^{-\theta'},e^{\frac{d^2\theta'}{d+1}t}\norm{w}^{-d\theta'}\}\\
            &\leq e^{\frac{d^2\theta'}{d+1}t}\min\{\norm{v}^{-\theta'},\norm{w}^{-d\theta'}\}\\
            &\leq e^{\frac{d^2\theta'}{d+1}t}\beta^{\theta'}(x),
        \end{split}    
        \end{equation}
       
        Combining \eqref{eq:beta case 1} and \eqref{eq:beta case 2}, we have
        \begin{equation*}
            \beta^{\theta'}(c_tu(s)x)\ll e^{\frac{\sigma d^2\theta'}{d+1}t}\beta^{\theta'}(x),
        \end{equation*}
        where the implied constant is independent of $t\geq0$ and $s\in I^d$.
    \end{proof}

    Let $l_1$ be as in \Cref{lem:back to compact}. Let $l\geq l_1$. We note that our choices of $\cN$ and the function $r_x^{\theta'}$ as in \Cref{lem:size of r_x} depend on $l$. 

    Let $\beta_\cN:X\to[1,\infty]$ be defined by 
    \begin{equation}\label{eq:definition of beta_N}
        \beta_\cN(x)=\alpha_\eps^{\theta}(x)+\beta^{\theta'}(x).
    \end{equation}
    
	\begin{prop} \label{prop:properties of beta}
	    The function $\bN$ satisfies the following properties:
        \begin{enumerate}
            \item  \label{itm:bN-lowersemi} $\bN$ is lower semicontinuous.
            \item \label{itm:bN-infty} For any $x\in X'$, $\bN(x)=\infty$ if and only if $x\in\cN$.
            \item \label{itm:bN-Hd-Lip} $\bN$ is Lipschitz with respect to the $H_d$-action; that is, for any compact set $\cC\subset H_d$, there exists $C>0$ such that $\bN(g x)\leq C\bN(x)$ for all $g\in \cC$ and $x\in X$.

            \item \label{itm:bN-contract} There exists $E\geq1$ such that for any $t\geq 0$ there exists $B_1(t)\geq 0$ such that for any $x\in X$ we have
                \begin{equation}\label{eq:contraction_beta}
		              \int_{I^d}\bN(c_{t}u(s)x)\de s \leq Ee^{-\frac{d}{d+1}\theta't}\bN(x)+B_1(t).
                \end{equation}

        \item \label{itm:bN-proper} For any $l'\geq 1$, the set $\bN^{-1}([1,l'])$ is compact. 
        \end{enumerate}
        
	\end{prop}

    \begin{proof}(Cf.~Proof of \cite[Theorem 6.4]{PSS23})
        
        Proof of (\ref{itm:bN-lowersemi}): By \cite[Lemma 4.1]{Shi20} we know that $\alpha^{\theta}$ is lower semicontinuous. So by \eqref{eq:r_x}, we have that  $x\mapsto r_x^{-\theta'}$ is lower semicontinuous.

        It remains to prove lower semicontinuity of $\beta^{\theta'}$. To prove this, let $x_i\to x$.  
        
        First suppose that $\beta^{\theta'}(x)=r_x^{-\theta'}$. So, 
        \[
        \liminf_{i\to\infty} \beta^{\theta'}(x_i)\geq \liminf_{i\to\infty} r_{x_i}^{-\theta'}\geq r_x^{-\theta'}=\beta^{\theta'}(x);
        \]
        where the first inequality holds because of \Cref{rem:beta-larger} and the second one holds due to lower semi-continuity of $y\mapsto r_y^{-\theta'}$. 

        Now suppose that $\beta^{\theta'}(x)>r_x^{-\theta'}$. Then there exists 
        $(v,w)\in \fq\times\fp$ such that  
        \[
        \exp(w)\exp(v)x\in\cN \text{ and } \beta^{\theta'}(x)=\min\{\norm{v}^{-\theta'},\norm{w}^{-d\theta'}\}>r_x^{-\theta'}.
        \]
        In particular, $\norm{v}<r_x<r_0$ and $\norm{w}<r_x^{1/d}<r_0$. 
        
        Let $g_i\to e$ such that $x_i=g_ix$. Then 
        \[
        x_i\in g_i\exp(-v)\exp(-w)\cN=g_i\exp(-v)\exp(-w)\exp(z)M_1y. 
        \]
        where $z\in \fs\cap B_r$. Therefore, by \Cref{rem:diffeo}, for all large $i$, there exist $v_i\to v$ in $\fq$, $w_i\to w$ in $\fp$, and $z_i\to z$ in $\fz$ such that 
        \[
        g_i\exp(-v)\exp(-w)\exp(z)M_1=\exp(-v_i)\exp(-w_i)\exp(z_i)M_1,
        \]
        and hence 
        $x_i\in\exp(-v_i)\exp(-w_i)\exp(O)M_1y$.

        Therefore, by \Cref{rem:beta-larger}, we have $\beta^{\theta'}(x_i)\geq \min\{\norm{v_i}^{-\theta'},\norm{w_i}^{-d\theta'}\}$. Hence,
        \[
        \liminf_{i\to\infty}\beta^{\theta'}(x_i)\geq \liminf_{i\to\infty}\min\{\norm{v_i}^{-\theta'},\norm{w_i}^{-d\theta'}\}=\min\{\norm{v}^{-\theta'},\norm{w}^{-d\theta'}\}=\beta^{\theta'}(x).
        \]
        Combining both cases, we proved that $\beta^{\theta'}$ is lower semicontinuous. 

        Proof of (\ref{itm:bN-infty}): By the definition of $X'$, for $x\in X'$ we have $\alpha^\theta(x)<\infty$, and it follows that $r_x<\infty$. Therefore, $\bN(x)=\infty$ if and only if $\exp(w)\exp(v)x\in\cN$ and $v=w=0$, that is, $x\in\cN$.
        
        Proof of (\ref{itm:bN-Hd-Lip}): By \cite[Lemma 4.1]{Shi20}, $X_l$ is compact and $\alpha^{\theta}$ is Lipschitz with respect to the action of $H_d$. It follows that $x\mapsto r_x^{-\theta'}$ is also Lipschitz with respect to the action of $H_d$. Since the action of $H_d$ on $\fq\times \fp$ by conjugation is Lipschitz, and $\cN$ is $H_d$-invariant, from definition it is straightforward to verify that $\beta^{\theta'}$ is also Lipschitz with respect to the action of $H_d$. Hence $\bN$ is Lipschitz with respect to the action of $H_d$.

        Proof of (\ref{itm:bN-contract}): Let $C_1\geq 1$, $C_2\geq 1$, $C_3\geq 1$, and $\eps_0>0$ be the constants in \eqref{eq:Lipschitz alpha}\eqref{eq:Lipschitz beta}\eqref{eq:beta Lipschitz}\eqref{eq:r_x}. Let $x\in X$ and $t\geq0$ be given.

        Firstly, suppose that there exists $v\in\mathfrak{q}$ and $w\in\mathfrak{p}$ such that $\exp(w)\exp(v)x\in\cN$ and
        \begin{equation} \label{eq:sharp}
            \min\{\norm{v}^{-\theta'},\norm{w}^{-d\theta'}\} > C_3e^{\left(\frac{d\theta'}{d+1}\right)t}\cdot C_1e^{d\left(\frac{\sigma d}{d+1}\theta'\right)t}\cdot \eps_0^{-\theta'}\alpha^{\left(\frac{\sigma d}{d+1}\theta'\right)}(x).
        \end{equation}
        Note that
        \begin{equation} \label{eq:temp3}
       C_3e^{\left(\frac{d\theta'}{d+1}\right)t}\cdot C_1e^{d\left(\frac{\sigma d}{d+1}\theta'\right)t}\cdot \eps_0^{-\theta'}\alpha^{\left(\frac{\sigma d}{d+1}\theta'\right)}(x)\geq \eps_0^{-\theta'}\alpha^{\frac{\sigma d\theta'}{d+1}}(x)=r_x^{-\theta'}.
        \end{equation} 
        Hence by the definition of $\beta^{\theta'}$, \eqref{eq:sharp}, \eqref{eq:temp3}, and \Cref{rem:beta-larger}, we have that 
        \begin{equation} \label{eq:sharp2}
        \beta^{\theta'}(x)=\min\{\norm{v}^{-\theta'},\norm{w}^{-d\theta'}\}.
        \end{equation}
        Then for every $s\in I^d$, we have
        \begin{equation*} \label{eq:temp1}
            \begin{split}
                &\min\{ \norm{c_tu(s)v}^{-\theta'}, \norm{c_tu(s)w}^{-d\theta'}\} \\
                &\geq C_3^{-1}e^{-\frac{d\theta'}{d+1}t}\min\{\norm{v}^{-\theta'},\norm{w}^{-d\theta'}\} \text{,  by \eqref{eq:beta Lipschitz} and \eqref{eq:beta Lipschitz 2},}\\ 
                &> C_1e^{d\left(\frac{\sigma d}{d+1}\theta'\right)t}\cdot
                \eps_0^{-\theta'}\alpha^{\left(\frac{\sigma d}{d+1}\theta'\right)}(x) \text{, by \eqref{eq:sharp},}\\
                &\geq \eps_0^{-\theta'}\alpha^{\left(\frac{\sigma d}{d+1}\theta'\right)}(c_tu(s)x) \text{, by \eqref{eq:Lipschitz alpha},}\\
                &=r_{c_tu(s)x}^{-\theta'} \text{, by \eqref{eq:r_x}. }
            \end{split}
        \end{equation*}
        Therefore, 
        \begin{equation} \label{eq:beta-val1}
        \beta^{\theta'}(c_tu(s)x)=\min\{ \norm{c_tu(s)v}^{-\theta'}, \norm{c_tu(s)w}^{-d\theta'}\}.
        \end{equation}
        
        Since $v\in\fq$ and $w\in \fp$, by \Cref{rem:Hd on p+q} and \cite[Lemma 3.1]{SY24TAMS},  we have
        \begin{equation} \label{eq:SY24PLMSLem3.1}
            \begin{split}
                \int_{I^d}\norm{c_tu(s)v}^{-\theta'}\de s &\ll e^{-\frac{d}{d+1}\theta't}\norm{v}^{-\theta'}  \quad\text{ and }\\
                \int_{I^d}\norm{c_tu(s)w}^{-d\theta'}\de s &\ll e^{-\frac{1}{d+1}(d\theta')t}\norm{w}^{-d\theta'}.
            \end{split}
        \end{equation}
        By \eqref{eq:sharp2}, \eqref{eq:beta-val1}, and \eqref{eq:SY24PLMSLem3.1} we have
        \begin{equation} \label{eq:contraction beta-theta'}
            \begin{split}
                \int_{I^d}\beta^{\theta'}(c_tu(s)x)\de s 
                &= \int_{I_d}\min\{\norm{c_tu(s)v}^{-\theta'},\norm{c_tu(s)w}^{-d\theta'}\}\de s \\ 
                &\leq \min\left\{\int_{I_d}\norm{c_tu(s)v}^{-\theta'}\de s, \int_{I_d}\norm{c_tu(s)w}^{-d\theta'}\de s\right\}\\
                &\ll e^{-\frac{d\theta'}{d+1}t}\min\{\norm{v}^{-\theta'},\norm{w}^{-d\theta'}\}=e^{-\frac{d\theta'}{d+1}t}\beta^{\theta'}(x).
            \end{split}
        \end{equation}
        On the other hand, by \Cref{lem:contraction_alpha},
        \begin{equation} \label{eq:contraction alpha 1}
            \int_{I^d}\alpha_\eps^{\theta}(c_tu(s)x)\de s\leq
		Ce^{-\theta t}\alpha_\eps^{\theta}(x)+b.
        \end{equation}
        Recall from \eqref{eq:theta and theta'} that $\frac{d\theta'}{d+1}<\theta$. Now \eqref{eq:contraction_beta} follows from \eqref{eq:definition of beta_N}, \eqref{eq:contraction beta-theta'}, and \eqref{eq:contraction alpha 1}.
        
        Secondly, suppose that
        \begin{equation} \label{eq:double sharp}
            \beta^{\theta'}(x)\leq C_2^{-1}e^{-\frac{\sigma d^2}{d+1}\theta't}\cdot C_1^{-1}e^{-\frac{d}{d+1}\theta't}\cdot \alpha^{\theta}(x).
        \end{equation}
        Then, for all $s\in I^d$ we have
        \begin{equation} \label{eq:bounding beta by alpha}
        \begin{split}
            \beta^{\theta'}(c_tu(s)x)&\leq C_2e^{\frac{\sigma d^2\theta'}{d+1}t}\beta^{\theta'}(x) \text{, by \eqref{eq:Lipschitz beta},}\\
            &\leq C_1^{-1}e^{-\frac{d\theta'}{d+1}t}\alpha^\theta(x) \text{, by \eqref{eq:double sharp},}\\
            &\leq e^{(\theta-\frac{d\theta'}{d+1})t}\alpha^\theta(c_tu(s)x)\text{, by \eqref{eq:Lipschitz alpha}.}
        \end{split}
        \end{equation}
        So,
        \[
        \begin{split}
            \int_{I^d}\beta_\cN(c_tu(s)x)\de s &=\int_{I^d}\alpha^\theta(c_tu(s)x)\de s +\int_{I^d}\beta^{\theta'}(c_tu(s)x)\de s \text{, by \eqref{eq:definition of beta_N},}\\
            &\leq(e^{(\theta-\frac{d\theta'}{d+1})t}+1)\int_{I^d}\alpha^\theta(c_tu(s)x)\de s \text{, by \eqref{eq:bounding beta by alpha},}\\
            &\ll e^{(\theta-\frac{d\theta'}{d+1})t}(e^{-\theta t}\alpha^\theta(x)+b) \text{, by \eqref{eq:contraction hypothesis alphaEpsTheta},}\\
            &< e^{-\frac{d\theta'}{d+1}t}\beta_\cN(x)+ be^{(\theta-\frac{d\theta'}{d+1})t}\text{, by \eqref{eq:definition of beta_N}.}
        \end{split}
        \]
        Hence \eqref{eq:contraction_beta} holds.

        Finally, suppose the negation of \eqref{eq:sharp} and the negation of \eqref{eq:double sharp} hold. Then,
        \begin{equation} \label{eq:temp2}
             C_1^{-1}C_2^{-1}e^{-\frac{\sigma d^2+d}{d+1}\theta't}\alpha^{\theta}(x) < \beta^{\theta'}(x) < \eps_0^{-\theta'}C_1C_3e^{\frac{\sigma d^2+d}{d+1}\theta't}\alpha^{\frac{\sigma d\theta'}{d+1}}(x).
        \end{equation}
        By \eqref{eq:choice of sigma} we have $\theta>\frac{\sigma d\theta'}{d+1}$. By the definition, we have $\alpha^\theta(x)\geq 1$. Hence \eqref{eq:temp2} implies that there exists $C_4>0$ such that 
        \[
        \alpha^\theta(x)<C_4e^{2\frac{\sigma d^2+d}{d+1}\theta'(1-\frac{\sigma d\theta'}{(d+1)\theta})^{-1}t}.
        \]
        Therefore, by \eqref{eq:Lipschitz alpha}, \eqref{eq:Lipschitz beta} and \eqref{eq:temp2}, there exist $C_5\geq 1$ and $A>0$, both independent of $t$, such that
        \begin{equation*}
            \int_{I^d}\beta_\cN(c_tu(s)x)\de s \leq C_5e^{At},
        \end{equation*}
        and this implies \eqref{eq:contraction_beta}.

        Thus,  \eqref{eq:contraction_beta} holds in all the three possible cases.

        Proof of (\ref{itm:bN-proper}): This follows because the inverse image of $[1,l']$ under $\alpha_\eps^\theta$ is compact, $\bN\geq \alpha_\eps^{\theta}$, and $\bN$ is lower semicontinuous. 
    \end{proof}

\subsubsection{Structure of $Z_G(F)$}
We have that $P_b^-=Z_G(\{b_t\})U_b^-$, see \eqref{eq:Pb-}, where 
    \begin{equation}
    \label{eq:U-}
    U^-_b=\{u\in G: \lim_{t\to\infty} b_tub_{-t}\to e\},
    \end{equation}
    and $U_b^-$ is an abelian group. 
 Moreover, $Z_G(\{b_t\})=H_dZ_G(H_d)$. 

\begin{lem} \label{lem:ZGF-Pb-}
   $Z_G(F)=Z_G(H_d)(U_b^-\cap Z_G(F))$. Moreover,
   \[
   Z_G(F)\cap U_b^-=(Z_G(F)\cap\exp(\fp))(Z_G(F)\cap U_b^-\cap M_1).
   \]
\end{lem} 

\begin{proof}
    We note that $F=\{a_t\}U$, $a_t=c_tb_t$ for all $t\in\R$. Let $t\geq 0$. Since $U$ is the horospherical subgroup of $H_d$ for $c_t$, for any finite dimensional representation of $H_d$, if a vector is fixed by $U$ then it is contained in the sum of expanding and zero eigen subspaces of $c_t$; for example, see~\cite[Lemma~5.2]{Shah-horo}. Therefore, if $g\in G$ is fixed by $U$ and $a_t=c_tb_t$ under conjugation, then $g$ is contained in the sum of contracting and zero eigen subspace for the action of $b_t$ by conjugation on the space of $(n+1)\times(n+1)$-real matrices. Hence by definition, $g\in P_b^-$.

    Now $\{b_t\}\subset Z_G(F)$, so $Z_G(F)=(Z_G(F)\cap Z_G(\{b_t\}))(Z_G(F)\cap U_b^-)$. Now $Z(H_d)\subset Z_G(F)$ and $H_d\cap Z_G(F)\subset Z(H_d)$. Therefore $Z_G(F)\cap Z_G(\{b_t\})=Z_G(H_d)$. 
\end{proof}

\subsubsection{Height function for \texorpdfstring{$\Omega\cN$}{\unichar{U+03A9}\unichar{U+1D4A9}} with \texorpdfstring{$a_tu(I^d)$}{a\textsubscript t u(I\textsuperscript d)}-contraction}
    Let $\Omega$ be a bounded open neighborhood of $0$ in \[
    \fp_1\defeq \Lie(Z_G(F))\cap \fp\subset \Lie(U_b^{-}).
    \]
    For any $x\in X$, we define 
	\begin{equation} \label{eq:definition of beta_N exponent}
	\rho_{\bN,\Omega}(x)=\limsup_{t\to\infty}\frac{\log \bigl(\sup_{\omega\in\Omega}\bN(b_t\exp(\omega) x)\bigr)}{t},
	\end{equation}
which measures exponential growth rate of the $\bN$ along the $b_t$-trajectory of $\Omega x$. 

    Let $d\omega$ denote the Lebesgue integral on $\fp_1$ normalized such that $\Omega$ has volume $1$. Let 
    \begin{equation} \label{eq:delta}
        0<\delta<\frac{d}{d+1}\theta',
    \end{equation} and $t_0\geq 0$. For any $x\in X$, define
	\begin{equation} \label{eq:definition_beta_tilde}
	\btt(x)=\int_{t_0}^{\infty}\int_{\omega\in\Omega} e^{-\delta t}\bN(b_t \exp(\omega) x)\de\omega\de t.
	\end{equation}

    \begin{prop} \label{prop:btt-finite}
    Let $x\in X$. Suppose that $\rho_{\bN,\Omega}(x)<\delta$ and $t_0\geq 0$ is such that 
    $\sup_{t\in [t_0,\infty)}\frac{\log \left(\sup_{\omega\in\Omega}\bN(b_t\exp(\omega) x)\right)}{t}<\infty$. Then $\btt(x)<\infty$.
    \end{prop}

    \begin{proof} 
Pick $\rho$ such that $\rho_{\bN,\Omega}(x)<\rho<\delta$. Then, there exists $t_1\geq t_0$ such that $\bN(b_t\exp(\omega) x)\leq e^{\rho t}$ for all $\omega\in\Omega$ and $t\geq t_1$. Now,
\[
\btt(x) = \int_{t_0}^{t_1}\int_{\omega\in\Omega} e^{-\delta t}\bN(b_t\exp(\omega) x)\de\omega\de t + \int_{t_1}^\infty \int_{\omega\in\Omega} e^{-\delta t}\bN(b_t\exp(\omega) x)\de\omega\de t.
\]
By our assumption on $t_0$ the integral over $[t_0,t_1]$ is finite, and the second integral is bounded by
$\int_{t_0}^\infty e^{-(\delta-\rho)t}\,\de t < \infty$.
\end{proof}

	\begin{prop} \label{prop:properties of beta tilde}
		\begin{enumerate}
			\item \label{itm:bt-semicont} $\btt$ is lower semicontinuous.

            \item \label{itm:bt-contract} There exists $C>0$ such that for any $t\geq 0$ there exists $B'(t)>0$ such that 
            \begin{equation} \label{eq:btt-contract}
                 \int_{I^d}\btt(a_{t}u(s)x)\de s \leq Ce^{-(\frac{d}{d+1}\theta'-\delta)t}\btt(x)+B'(t), \quad \forall x\in X. 
            \end{equation}
            \item \label{itm:bt-lipschitz} $\btt$ is Lipschitz with respect to the action of the semigroup  
            $\{a_t\}_{t\geq 0}U$.
            \item \label{itm:bt-infinite} If $x\in \exp(\Omega)\cN$, then $\btt(x)=\infty$. 
            \item \label{itm:bt-proper}For any $M>0$, $\btt^{-1}([1,M])$ is compact.

		\end{enumerate}
	\end{prop}

    We will prove each part of the proposition separately. In the proof, we will use the fact that $\Omega\subset Z_G(F)$ commutes with $a_tu(s)$ without mentioning it.
    
\begin{proof}[Proof of (\ref{itm:bt-semicont})] Since $\bN$ is lower semicontinuous by \Cref{prop:properties of beta}, for any sequence $x_i\to x$ in $X$, we have
\[
\begin{split}
\btt(x) &= \int_{t_0}^\infty e^{-\delta t}\int_{\omega\in\Omega} \bN(b_t\exp(\omega) x)\de\omega\de t \\
&\leq \int_{t_0}^\infty \int_{\omega\in\Omega} \liminf_{i\to\infty} e^{-\delta t}\bN(b_t\exp(\omega) x_i)\de\omega\de t \\
&\leq \liminf_{i\to\infty} \int_{t_0}^\infty \int_{\omega\in\Omega} e^{-\delta t}\bN(b_t\exp(\omega) x_i)\de\omega\de t \quad \text{(by Fatou's lemma)}\\
&= \liminf_{i\to\infty} \btt(x_i).
\end{split}
\]
Hence $\btt$ is lower semicontinuous.
\end{proof}

\begin{proof}[Proof of (\ref{itm:bt-contract})] 

Put
\[
\lambda=\frac{d}{d+1}\theta'.
\]
We shall prove that, for every $\tau\geq 0$, there is $B'(\tau)>0$ such that
\[
\int_{I^d}\widetilde\beta_\Omega(a_\tau u(s)x)\,ds
\leq
E e^{-(\lambda-\delta)\tau}\widetilde\beta_\Omega(x)+B'(\tau),
\]
where $E$ is the constant in Proposition~3.8(4). 

All integrands below are non-negative, so Tonelli's theorem justifies the changes
in the order of integration. Let $x\in X$ and $\tau\geq 0$. By definition,
\[
\begin{aligned}
\int_{I^d}\widetilde\beta_\Omega(a_\tau u(s)x)\,ds
&=
\int_{I^d}\int_{t_0}^{\infty}\int_{\Omega}
e^{-\delta t}
\beta_N\!\left(b_t\exp(\omega)a_\tau u(s)x\right)
\,d\omega\,dt\,ds .
\end{aligned}
\]
Since $\omega\in\fp_1\subset \operatorname{Lie}(Z_G(F))$ and
$a_\tau u(s)\in F$, we have
\[
\exp(\omega)a_\tau u(s)=a_\tau u(s)\exp(\omega).
\]
Moreover $a_\tau=b_\tau c_\tau$, and $b_t$ commutes with $c_\tau u(s)\in H_d$.
Therefore
\[
\begin{aligned}
b_t\exp(\omega)a_\tau u(s)
&=
b_ta_\tau u(s)\exp(\omega)  \\
&=
b_tb_\tau c_\tau u(s)\exp(\omega) \\
&=
c_\tau u(s)b_{t+\tau}\exp(\omega).
\end{aligned}
\]
Hence
\[
\begin{aligned}
\int_{I^d}\widetilde\beta_\Omega(a_\tau u(s)x)\,ds
&=
\int_{t_0}^{\infty}\int_{\Omega}
e^{-\delta t}
\left(
\int_{I^d}
\beta_N\!\left(c_\tau u(s)b_{t+\tau}\exp(\omega)x\right)
\,ds
\right)
d\omega\,dt .
\end{aligned}
\]
Applying Proposition~3.8(4) with
\[
y=b_{t+\tau}\exp(\omega)x
\]
gives
\[
\int_{I^d}
\beta_N\!\left(c_\tau u(s)b_{t+\tau}\exp(\omega)x\right)
\,ds
\leq
E e^{-\lambda \tau}\beta_N\!\left(b_{t+\tau}\exp(\omega)x\right)
+B_1(\tau).
\]
Therefore,
\[
\begin{aligned}
\int_{I^d}\widetilde\beta_\Omega(a_\tau u(s)x)\,ds
&\leq
E e^{-\lambda\tau}
\int_{t_0}^{\infty}\int_{\Omega}
e^{-\delta t}
\beta_N\!\left(b_{t+\tau}\exp(\omega)x\right)
\,d\omega\,dt  \\
&\qquad
+
B_1(\tau)\int_{t_0}^{\infty}\int_{\Omega}e^{-\delta t}\,d\omega\,dt .
\end{aligned}
\]
Since $d\omega$ is normalized so that $\Omega$ has volume $1$, the second term is
\[
B_1(\tau)\int_{t_0}^{\infty}e^{-\delta t}\,dt
=
\delta^{-1}e^{-\delta t_0}B_1(\tau).
\]
For the first term, change variables $t'=t+\tau$. Then
\[
\begin{aligned}
\int_{t_0}^{\infty}\int_{\Omega}
e^{-\delta t}
\beta_N\!\left(b_{t+\tau}\exp(\omega)x\right)
\,d\omega\,dt
&=
e^{\delta\tau}
\int_{t_0+\tau}^{\infty}\int_{\Omega}
e^{-\delta t'}
\beta_N\!\left(b_{t'}\exp(\omega)x\right)
\,d\omega\,dt'  \\
&\leq
e^{\delta\tau}\widetilde\beta_\Omega(x).
\end{aligned}
\]
Combining the last two estimates, we obtain
\[
\int_{I^d}\widetilde\beta_\Omega(a_\tau u(s)x)\,ds
\leq
E e^{-(\lambda-\delta)\tau}\widetilde\beta_\Omega(x)
+
\delta^{-1}e^{-\delta t_0}B_1(\tau).
\]
Thus the desired estimate holds with
\[
C=E,
\qquad
B'(\tau)=\delta^{-1}e^{-\delta t_0}B_1(\tau).
\]
If one wants $B'(\tau)>0$ strictly, replace this by
$B'(\tau)=1+\delta^{-1}e^{-\delta t_0}B_1(\tau)$.

\end{proof}

\begin{proof}[Proof of (\ref{itm:bt-lipschitz})] Let $T\geq 0$ and $\cC$ be a compact subset of $U$. By (\ref{itm:bN-Hd-Lip}) of \Cref{prop:properties of beta}, there exists $C_1\geq 1$ such that
\[
\bN(c_tux)\leq C_1\bN(x), \, \forall t\in [0,T],\, u\in \cC, \, \forall x\in X.
\]
Let $t_1\in [0,T]$ and $u\in \cC$, and $h=a_{t_1}u\in F$. 
Then,
\begin{align*}
\btt(hx) &= \int_{\omega\in\Omega} \int_{t_0}^\infty e^{-\delta t}\bN(b_t\exp(\omega) hx)\de t\de\omega \\
&= \int_{\omega\in\Omega} \int_{t_0}^\infty e^{-\delta t}\bN(hb_t\exp(\omega) x)\de t \de\omega 
\text{, as $\{b_t\}\exp(\Omega)\subset Z_G(F)$,}\\
&=\int_{\omega\in\Omega} \int_{t_0}^\infty e^{-\delta t}\bN(c_{t_1}ub_{t+t_1} \exp(\omega)x)\de t \de\omega
\text{, as $a_{t_1}=c_{t_1}b_{t_1}$,}\\
&\leq C_1 \int_{\omega\in\Omega}\int_{t_0}^\infty e^{-\delta t}\bN(b_{t+t_1} \exp(\omega)x)\de t \de\omega\\
&=C_1 e^{\delta t_1}\int_{t_0+t_1}^\infty \int_{\omega\in\Omega} e^{-\delta t'}\bN(b_{t'}\exp(\omega)x)\de\omega\de t'\\
&\leq C_1e^{\delta T}\btt(x). 
\end{align*}
\end{proof}

\begin{proof}[Proof of (\ref{itm:bt-infinite})]
Let $x\in \exp(\Omega)\cN$. Let $\omega_1\in\Omega$ such that $x\in \exp(\omega_1)\cN$. 

For any $t\geq 0$, and $w\in \fp$, in view of \eqref{eq:bt},
\[
b_t\exp(w)\exp(-\omega_1)x=\exp(\Ad(b_t)(w))b_t\exp(-\omega_1)x=\exp(e^{-\frac{1}{d+1}t}w)b_t\exp(-\omega_1) x.
\]
Now, by \eqref{eq:adg1v0}, $\{b_t\}\subset M_1$ and $Z_G(H_d)\subset Z_G(\{b_t\})$. Hence $\cN$ is $\{b_t\}$-invariant. So, $b_t\exp(-\omega_1) x\in \cN$, and $e^{-\frac{1}{d+1}t}w\in\fp$. Therefore, by \Cref{rem:beta-larger}, we have
\begin{align}
\beta^{\theta'}(b_t\exp(w)\exp(-\omega_1) x)
\geq 
\norm{e^{-\frac{1}{d+1}t}w}^{-d\theta'}\geq e^{\frac{d\theta'}{d+1}t}\norm{w}^{-d\theta'}.
\label{eq:bt-contract-p}
\end{align}
 Therefore, 
\begin{align*}
\int_{\omega\in\Omega} \bN(b_t\exp(\omega) x)\de\omega
&\geq \int_{\omega\in\Omega} \beta^{\theta'}(b_t\exp(\omega) x)\de\omega\\
&\geq \int_{\omega\in \Omega} \beta^{\theta'}(b_t\exp(\omega+\omega_1)\exp(-\omega_1) x)\de \omega\\
&\geq C^{-d\theta'}e^{\frac{d\theta'}{d+1}t}
\text{, by \eqref{eq:bt-contract-p},}
\end{align*}
where $C=\sup\{\norm{\omega+\omega_1}:\omega\in \Omega\}<\infty$.
Therefore, 
\begin{align*}
\btt(x)
&\geq \int_{t_0}^\infty e^{-\delta t}\left(\int_{\omega\in\Omega} \bN(b_t\exp(\omega) x)\de\omega \right)\de t\\
&\geq C^{-d\theta'}\int_{t_0}^{\infty} e^{(\frac{d\theta'}{d+1}-\delta)t}\de t=\infty,
\end{align*}
as by \eqref{eq:delta}, $\delta<\frac{d\theta'}{d+1}$.
\end{proof}
\begin{proof}[Proof of (\ref{itm:bt-proper})] Let $M\geq 0$. Then $\btt^{-1}([0,M])$ is closed, because $\btt$ is lower semicontinuous and nonnegative. For pre-compactness, note that for any $y\in X$:
\[
\btt(y) \geq \int_{t_0}^{t_0+1} \int_{\omega\in\Omega} e^{-\delta t}\bN(b_t \exp(\omega)y)\de\omega\de t \geq e^{-\delta}\min_{t_0\leq t\leq t_0+1}\min_{\omega\in\Omega} \bN(b_t\exp(\omega) y).
\]
By (\ref{itm:bN-proper}) of \Cref{prop:properties of beta}, $\bN^{-1}([0,e^\delta M])$ is a compact subset of $X$, and $\{b_t:t\in[t_0,t_0+1]\}$ and $\overline{\exp(\Omega)}$ are compact subsets of $G$. So, the set 
\[
\{y\in X:\min_{t_0\leq t\leq t_0+1}\min_{\omega\in\Omega}e^{-\delta}\bN(b_t\exp(\omega) y)\leq 
M\} 
\]
is compact, and contains $\btt^{-1}([0,M])$. 
\end{proof}

\section{Obstruction to avoiding a singular submanifold} \label{sec:obstruction}
	Let $x\in X$. For any measurable subset $K$ of $X$ and $T>0$, define
	\begin{equation*}
	\mathcal{C}_K^T(s)= \frac{1}{T}\int_0^T\mathbf{1}_K(a_tu(s)x)\de t,
	\end{equation*}
	where $\mathbf{1}_K$ is the characteristic function of $K$.

Let the notation be as in Section~\ref{sec:Mf}. Let $l_1$ be as in Lemma~\ref{lem:back to compact} and we fix $l\geq l_1$. Let $O$ be a bounded open neighborhood of $0$ in $\fs\subset \Lie(Z_G(H_d))$ in \Cref{lem:size of r_x}. Let $\Omega$ be any bounded neighborhood of $0$ in $\fp_1=\Lie(Z_G(F))\cap \fp$. Then $\exp(\Omega)$ is a relatively compact subset of $P_b^-$. 

In this section, for $g\in G$ and $v\in\fg$, sometimes we simply write $gv$ instead of $(\Ad g) v$ to shorten the expressions.

    \begin{prop}\label{prop:continuous_exponential_recurrence&avoidence}
      Suppose $\rho_{\bN,\Omega}(x)<\frac{d}{d+1}\theta'$. Then, for any $0<\epsilon<1$, there exists a compact set $K$ in $X$ with $K\cap \exp(\Omega)\cN=\emptyset$ and real numbers $0<a<1,C\geq 1$ such that
			\begin{equation} \label{eq:continuous_exponential_recurrence2}
		\abs{ \{ s\in I^d \colon \mathcal{C}_K^T(s)\leq 1-\epsilon \} } \leq Ca^T.
		\end{equation}
    \end{prop}
    \begin{proof}
        Choose $\delta$ in \eqref{eq:delta} such that $\rho_{\bN,\Omega}(x)<\delta<\frac{d}{d+1}\theta'$. We can pick $t_0\geq 0$ be such that $\sup_{t\geq t_0} \frac{\log\sup_{\omega\in\Omega}\bN(b_t\exp(\omega)x)}{t}<\infty$.  For these choices of $\delta$ and $t_0$, we define $\btt$ as in \eqref{eq:definition_beta_tilde}. By \Cref{prop:btt-finite}, we have $\btt(x)<+\infty$. 

        And by (\ref{itm:bt-infinite}) of \Cref{prop:properties of beta tilde} we have $\btt(y)=\infty$ for all $y\in \exp(\Omega)\cN$. Then the proof of the desired conclusion is analogous to the proof of \Cref{prop:continuous_exponential_recurrence}, where we replace $\at$ by $\btt$ and use properties of $\btt$ noted in \Cref{prop:properties of beta tilde}. 
    \end{proof}

    \begin{cor} \label{cor:non focusing}
         Let $x\in X$. Let $z\in Z_G(F)$ and suppose that $\rho_{\bN,\Omega}(zx)<\frac{d}{d+1}\theta'$. Then given a sequence $T_i\to\infty$, for almost all $s\in I^d$, any subsequence of the sequence $\bigl\{\frac{1}{T_i}\int_0^{T_i} a_tu(s)\delta_{x}\de t\bigr\}_{i\in\N}$ has a limit point, say $\nu$ in $\cP(X)$, such that $\nu(z\exp(\Omega)\cN)=0$, where $\cP(X)$ denotes the space of probability measures on $X$ endowed with the weak-$\ast$ topology.   
    \end{cor}
    
    \begin{proof}
        It follows from \Cref{prop:continuous_exponential_recurrence&avoidence} and Borel-Cantelli lemma that
        there exists a null set $E\subset I^d$,
        such that for any $s\in I^d\setminus E$, and any subsequence of $\bigl\{\frac{1}{T_i}\int_0^{T_i} a_tu(s)\delta_{zx}\de t\bigr\}_{i\in\N}$ has a limit point, say $\nu$, in $\cP(X)$ such that $\nu(\exp(\Omega)\cN)=0$. The proof is analogous to the deduction of \Cref{thm:main_nondivergence} from \Cref{prop:continuous_exponential_recurrence}.  
        Now for any $s\in I^d\setminus E$,
        since $z$ commutes with $\{a_t\}u(s)$ and $\delta_{zx}=z_\ast\delta_x$, we have
        \[
        \frac{1}{T_i}\int_0^{T_i} a_tu(s)\delta_{x}\de t=(z^{-1})_\ast\left(\frac{1}{T_i}\int_0^{T_i} a_tu(s)\delta_{zx}\de t\right), \,\forall i\in\N.
        \]
        Therefore any subsequence of this sequence has a limit point of the form $(z^{-1})_\ast\nu$, where $\nu\in \cP(X)$ such that $\nu(\exp(\Omega)\cN)=0$. Therefore $(z^{-1})_\ast\nu(z\exp(\Omega)\cN)=\nu(\exp(\Omega)\cN)=0$.
    \end{proof}

    \begin{remark}
        As noted in \Cref{rem:alternate}, one can also use the techniques developed in \cite{KKLM17,Kha20} to prove \Cref{cor:non focusing} using \Cref{prop:properties of beta tilde}.
    \end{remark}

\begin{cor}
    \label{cor:general} 
    Let $x\in X$ and suppose that $\rho_{\bN,\Omega}(zx)<\frac{d}{d+1}\theta'$ for all $z\in Z_G(F)$. Then given a sequence $T_i\to\infty$, for almost all $s\in I^d$, any subsequence of the sequence $\bigl\{\frac{1}{T_i}\int_0^{T_i} a_tu(s)\delta_{x}\de t\bigr\}_{i\in\N}$ has a limit point, say $\nu$ in $\cP(X)$, such that $\nu(N(F,M)x_0)=0$. 
\end{cor}

\begin{proof}
    We have that $Z_G(F)M_1=\cup_{z\in Z_G(F)} z\exp(\Omega)\exp(O)M_1$. Due to \Cref{lem:ZGF-Pb-}, 
    $\Lie(Z_G(F))=\fp_1\oplus \fs\oplus \Lie(Z_G(F)\cap M_1)$. Therefore, $z\exp(\Omega)\exp(O)M_1$ is an open subset of $Z_G(F)M_1$ for each $z\in Z_G(F)$. By second countability, we pick countably many $z_i\in Z_G(F)$ such that $Z_G(F)M_1=\cup_{i=1}^\infty z_i\exp(\Omega)\exp(O)M_1$. 
    
    We picked any $p\in N_G(M_1)$, and we set $y=pg_1$, $Y=M_1y$, and $\cN=\exp(O)Y$. So we have $Z_G(F)Y=\cup_{i=1}^\infty z_i\exp(\Omega)\cN$. Due to \Cref{cor:non focusing}, we deduce that for almost all $s\in I^d$, $\nu(Z_G(F)Y)=0$ for any of the limiting measures as in the proposition. We note that $Z_G(F)Y=Z_G(F)M_1pg_1\Gamma/\Gamma\subset X$. 

    Further, by \Cref{prop:ZFM} we have $N(F,M)=\cup_{i=1}^{s_1} Z_G(F)M_1p_ig_1$, where $p_i\in N_G(M_1)$. Therefore, the conclusion of the proposition follows from the above observations. 
\end{proof}

\subsection{Consequence of large exponent}

    Let $\bfG=\SL_{n+1}$, and let $\alpha_i$ be the root $\diag(t_j)\mapsto t_it_{i+1}^{-1}$ for $1\leq i\leq n$. Then $\Delta=\left\{ -\alpha_i \right\}_{1\leq i\leq n}$ form a set of simple roots. Let $\bfP_i$ be the maximal parabolic subgroup of $\bfG$ associated to $\Delta\setminus\{-\alpha_i\}$ for $1\leq i\leq n$, and let $P_i=\bfP_i(\bbR)$. In block matrix form,
	\begin{equation} \label{eq:Pi}
	P_i = \left\{
	\begin{pmatrix}
	A_{i\times i} & 0_{i\times(n+1-i)} \\
	C_{(n+1-i)\times i} & D_{(n+1-i)\times(n+1-i)}
	\end{pmatrix}
	\right\}.
    \end{equation}

    Recall that $\cA$ is the $d$-dimensional affine subspace of $U_a^+\cong \R^n$ parametrized by $A$.
    We suppose that $\omega(A)<n$. We fix any $\theta\in (0,\frac{d}{d+1})$. Then, by \cite[Proposition 4.1]{SY24TAMS}, we have $0\leq \rho_{\eps,\theta}(x_A)<\theta$. 
    We choose $\theta'$ as in \eqref{eq:theta and theta'} and  $\sigma>1$ as in \eqref{eq:choice of sigma}, and then $\rho_1$ such that 
    \begin{equation} \label{eq:choice of sigma2}
        \rho_{\eps,\theta}(A) <\rho_1<\frac{d\theta'}{d+1}<\sigma\frac{d\theta'}{d+1}<\theta <\frac{d}{d+1}.
    \end{equation}
    
    \begin{prop} \label{prop:consequence of large exponent}
         Suppose that $\rho_{\bN,\Omega}(zx_A)\geq \frac{d}{d+1}\theta'$ for some bounded open neighborhood $\Omega$ of $0$ in $\fp_1$ and for some $z\in Z_G(F)$. Then $\mathcal{A}$ satisfies possibility~(\ref{item:algebraic obstruction}) of \Cref{thm:main_equidistribution}. 
    \end{prop}

    \begin{proof}  
       By \eqref{eq:definition of beta_N} and \eqref{eq:definition of beta_N exponent}, there exists $t_i\to\infty$, $\omega_i\in\exp(\Omega)z\subset Z_G(F)$, and $\xi_i\uparrow \frac{d}{d+1}$ such that 
       \[
       \beta_\cN(b_{t_i}\omega_i x_A) =\alpha_{\eps}^\theta(b_{t_i}\omega_ix_A)+\beta^{\theta'}(b_{t_i}\omega_ix_A)>e^{\xi_i\theta't_i}.
       \]
       Therefore, by \Cref{prop:bt-Omega}, 
       $\rho_{\eps,\theta,\exp(\Omega)z}(x_A)=\rho_{\eps,\theta}(x_A)<\rho_1$. 
       Therefore, after passing to a subsequence, since $\rho_1<\xi_i\theta'$,
       \begin{equation} \label{eq:alpha-bound}
       \alpha_\eps^\theta(b_{t_i}\omega_i x_A)\leq e^{\rho_1t_i}< e^{\xi_i\theta' t_i}.
       \end{equation}
       Hence, by \eqref{eq:choice of sigma2} we deduce that 
       \begin{equation} \label{eq:beta-bound}
       \beta^{\theta'}(b_{t_i}\omega_ix_A)>e^{\xi_i\theta' t_i}.
       \end{equation}
       Now, recalling \eqref{eq:r_x}, by \eqref{eq:alpha-bound} we have 
       \begin{equation}
       r_{b_{t_i}\omega_ix_A}^{-\theta'}
       =(\eps_0\alpha_\eps^{\sigma\frac{d}{d+1}}(b_{t_i}\omega_ix_A))^{\theta'}
       \leq  \alpha_\eps^\theta(b_{t_i}\omega_ix_A)^{\frac{\sigma d\theta'}{(d+1)\theta}}< e^{\frac{\sigma\rho_1}{\theta}\frac{d\theta'}{d+1}t_i}.  
       \label{eq:r-bound}
       \end{equation}
     By \eqref{eq:choice of sigma2}, we have $\sigma\rho_1<\sigma\frac{d\theta'}{d+1}<\theta$, so $\frac{\sigma\rho_1}{\theta}<1$. Therefore, by \eqref{eq:beta-bound} and \eqref{eq:r-bound}, since $\xi_i\uparrow \frac{d}{d+1}$, after passing to a subsequence, we get
       \[
       \beta^{\theta'}(b_{t_i}\omega_ix_A)>r_{b_{t_i}\omega_ix_A}^{-\theta'},\ \forall i.
       \]
       Therefore, by the definition \eqref{eq:defbeta} of $\beta^{\theta'}$, for every sufficiently large $i$, there exists $(v_i,w_i)\in\mathfrak{q}\times\mathfrak{p}$ such that
       \begin{equation} \label{eq:beta-vw}
            \exp(w_i)\exp (v_i)b_{t_i}\omega_ix_A\in\cN \text{ and } \beta^{\theta'}(b_{t_i}\omega_ix_A)\leq \min\bigl\{\norm{v_i}^{-\theta'},\norm{w_i}^{-d\theta'}\bigr\},
       \end{equation}

        Recall that by \eqref{eq:Y} and \eqref{eq:cN-def}, \[
        \cN=\exp(O)Y=\exp(O)pg_1Mx_0=\exp(O)M_1pg_1x_0
        \text{ and } x_A=g_Ax_0.
        \]
        Hence there exists $\gamma_i\in\Gamma=\Stab_G(x_0)$ such that
        \begin{equation} \label{eq:gammai-choice}
            b_{t_i}\omega_ig_A\gamma_i\in \exp(-v_i)\exp(-w_i)\exp(O)M_1pg_1.
        \end{equation}
       Hence, since $\{b_t\}$ commutes with elements of $Z_G(H_d)$ and $\{b_t\}\subset M_1$, we have
       \begin{equation} \label{eq:gAgammai}
           \omega_ig_A\gamma_i\in \exp(-\Ad(b_{-t_i})v_i)\exp(-\Ad(b_{-t_i})w_i)\exp(O)M_1pg_1. 
       \end{equation}
       By \eqref{eq:beta-bound} and \eqref{eq:beta-vw},
       \[
       \norm{v_i}\ll e^{-\xi_it_i} \text{ and }\norm{w_i}\ll e^{-\frac{\xi_i}{d}t_i}.
       \]
       And, since $v_i\in\fq$ and $w_i\in \fp$, in view of \eqref{eq:bt} and \eqref{eq:p}, we get   
       \begin{equation} \label{eq:b-tviwi}
       \begin{array}{ll}
       \norm{\Ad(b_{-t_i})v_i}&\ll e^{-\frac{1}{d+1}t_i}\norm{v_i}\ll e^{-(1-\epsilon_i)t_i} \\
       \norm{\Ad(b_{-t_i}) w_i}&\ll e^{\frac{1}{d+1}t_i}\norm{w_i}\ll e^{(\epsilon_i/d)t_i},
       \end{array}
       \end{equation}
       where $\epsilon_i=\frac{d}{d+1}-\xi_i\downarrow 0$. Let $Y_i=\Ad(g_A^{-1} \omega_i^{-1} b_{-t_i})(v_i)$ for each $i$, and it follows that
       \begin{equation} \label{eq:size of Y_i}
           \norm{Y_i}\ll e^{-(1-\epsilon_i)t_i}.
       \end{equation}
       We have $\omega_i\in Z_G(F)\subset P_{d+1}$ and $\exp(O)\subset Z_G(H_d)\subset P_{d+1}$. Moreover due to the initial $(r\times r)$-block form of $M_1$ and that $r\geq d+1$,  $P_{d+1}M_1=P_{d+1}P_{r}$. Therefore,
       \begin{equation} \label{eq:gA-in}
       g_A\exp(Y_i)\in\omega_i^{-1}\exp(-\Ad(b_{-t_i})w_i)\exp(O)M_1pg_1\gamma_i^{-1}\subset P_{d+1}P_{r}pg_1\gamma_i^{-1}, \quad \forall i.
       \end{equation}
      Moreover, we recall that $M_1pg_1=pg_1 M$ and $M=\Stab_G(v_0)$. So, for each $i$, we have
       \begin{equation}
           \label{eq:li}
           l_i\in g_A^{-1}\omega_i^{-1}\exp(-\Ad(b_{-t_i})v_i)\exp(-\Ad(b_{-t_i})w_i)\exp(O)pg_1\in G
       \end{equation}
       and $m_i\in M$ such that
       \begin{align}
           \label{eq:gammaili}
           \gamma_i=l_im_i \text{ and } \norm{l_i}\ll e^{(\epsilon_i/d)t_i}. 
       \end{align}
    In view of reduction theory for the lattice $M\cap\Gamma$ acting on the left on $M$, by passing to a subsequence, there exists a generalized Siegel domain $\fS$ in $M$ and $\lambda_i\in M\cap\Gamma$ such that $m_i^{-1}\in \lambda_i\fS$ for all $i$. Then, $\gamma_i\lambda_i\in l_i\fS^{-1}$. Then, by \cite[Theorem~1.1]{DO22} there exist $A_1,C_1>0$ such that 
    \begin{equation} \label{eq:gammai-size}
       \norm{\gamma_i\lambda_i}\leq C_1\norm{l_i}^{A_1}.
    \end{equation}
    To see how this follows from the quoted result, consider the action of $m\in M$ on the vector space $\Mat(n+1,\R)$ by right multiplication by $m^{-1}$. And note that $\Gamma\subset \Mat(n+1,\Z)$ and the action of $l_i$ on $\Mat(n+1,\R)$ by left multiplication commutes with the above right action of $M$. 

    Noting that 
    \[
    M_1pg_1\lambda_i=pg_1M\lambda_i=pg_1M=M_1pg_1,
    \]
    replacing $\gamma_i$ by $\gamma_i\lambda_i$ in \eqref{eq:gammai-choice}, without loss of generality, we may assume that
    \begin{equation} \label{eq:gammai-growth}
        \norm{\gamma_i}\ll e^{\epsilon'_it_i}\text{, where $\epsilon'_i\to 0$ as $i\to\infty$.}
    \end{equation}

    We recall the notation from \Cref{def:F-Liouville}. Let $\overline{\cA}$ denote the projective closure of $\cA$ in $\bbP^n$. Then it follows that $\overline{\cA}$ is parametrized by $P_{d+1}g_A\in P_{d+1}\backslash G \cong \Gr(d+1,n+1)$. Let $\cT_i$ denote the $(r-1)$-dimensional subspace of $\bbP^n$ parametrized by $P_rpg_1\gamma_i^{-1}$. Since right multiplication by $\gamma_i^{-1}$ does not change the field of definition, we have $\cT_i$ is defined over $\sigma(\bbF)$, where $\sigma:\bbF\hookrightarrow\R$ is a real embedding determined by $p$. Indeed, in view of the choice of $p=p_j$ for some $1\leq j\leq s_1$ as in \Cref{prop:ZFM} and by \Cref{rem:g1-F} we have that the first $r$-rows of $p_jg_1$ have entries in $\sigma_j(\bbF)$. By replacing $\bbF$ with $\sigma(\bbF)$, without loss of generality we may assume that $\cT_i$ is defined over $\bbF$. Moreover, there exist constants $C>0$ and $A>0$ such that 
    \begin{equation}\label{eq:T_i height}
        H_{\cL}(\cT_i)\leq C\norm{\gamma_i}^A\ll e^{A\epsilon'_it_i},
    \end{equation}
    which can be verified by considering the Pl\"ucker embedding $\Gr(r,n+1)\hookrightarrow\bbP(\wedge^{r}V)$ and noting that $\cL$ is the pull-back of $\cO(1)$. 
    
    By \eqref{eq:gA-in}, we have $P_1P_{d+1}g_A\exp(Y_i)\subset P_1P_rpg_1\gamma_i^{-1}$, or in other words, we have the inclusion
    $\overline{\cA}\exp(Y_i)\subset \cT_i$ in $\bbP^n\cong P_1\backslash G$. Therefore, by \eqref{eq:size of Y_i},
    \begin{equation}
        d(\overline{\cA},\cT_i)\leq d(\overline{\cA},\overline{\cA}\exp(Y_i))\ll \norm{Y_i}\ll e^{-(1-\epsilon_i)t_i}. 
    \end{equation}
    Now $\epsilon_i\to 0$ and $\epsilon'_i\to 0$ as $i\to\infty$. Therefore, by \eqref{eq:T_i height}, given any $\kappa>1$, 
    \begin{equation*}
        d(\overline{\cA},\cT_i)\leq H_\cL(\cT_i)^{-\kappa},\quad\forall i\gg 1.
    \end{equation*}
    Hence by \Cref{def:F-Liouville} we conclude that $\overline{\cA}$ is $(\bbF,r-1)$-Liouville, and thus $\cA$ is $(\bbF,r-1)$-Liouville.
 
    \end{proof}

Though not needed for our proof, we note the following. 

\begin{cor} \label{cor:inFSubspace}
    Suppose $\rho_{\bN,\Omega}(zx_A)> \frac{d}{d+1}\theta'$ in \Cref{prop:consequence of large exponent}. Then $\cA$ is contained in a $(r-1)$-dimensional affine subspace defined over $\bbF$.  
\end{cor}

\begin{proof}
 We can choose $\frac{\rho_{\bN,\Omega}(zx_A)}{\theta'}>\xi_i=\xi>\frac{d}{d+1}$ and we can obtain \eqref{eq:gAgammai} and \eqref{eq:b-tviwi} with $\eps_i=\frac{d}{d+1}-\xi<0$. Therefore, after passing to a subsequence we get $\gamma_i=\gamma_1$ for all $i$.  Therefore $\cT_i=\cT_1$ for all $i$. Since $Y_i\to 0$, we get $\overline{\cA}\subset \cT_1$.
 \end{proof}

\subsection{Proof of the main theorem}
    \begin{proof}[Proof of \Cref{thm:main_equidistribution}]
        Under the assumption $\omega(A)<n$, suppose possibility~(\ref{item:algebraic obstruction}) of \Cref{thm:main_equidistribution} does not occur, then by \Cref{prop:consequence of large exponent}, for every $z\in Z_G(F)$ and any relatively compact open $\Omega$, we have $\rho_{\bN,\Omega}(zx_A) < \frac{d}{d+1}\theta'$. Then due to \eqref{eq:NFM}, the \Cref{cor:general} completes the proof of \Cref{prop:goal}. So by the analysis in the beginning of \Cref{subsect: avoidance of singular sets} we can obtain \eqref{eq:equidistribution}, which is equivalent to (\ref{item:equidistribution}) of \Cref{thm:main_equidistribution}.
    \end{proof}

\subsection*{Acknowledgment} 
We thank Ronggang Shi for enlightening discussions, in particular for suggesting the reference \cite{PSS23}, which led us to use the Margulis function with respect to the cusp and singular sets simultaneously.

Part of this research was performed while an author (N.S.) was visiting the Simons Laufer Mathematical Sciences Institute (SLMath), which is supported by the National Science Foundation (Grant No. DMS-2424139). 

	\bibliographystyle{alpha}
	\bibliography{bibfile}
	
\end{document}